\documentclass[12pt]{article}
\usepackage{epsf}
\usepackage{amsbsy,amsmath}
\usepackage{mathtools}
\usepackage{mathrsfs}
\usepackage{amsfonts}
\usepackage{amssymb}
\usepackage{enumitem}
\usepackage{eucal}
\usepackage{graphics,mathrsfs}
\usepackage{amsthm}
\usepackage{secdot}
\usepackage{latexsym}
\usepackage{graphicx}
\usepackage{wrapfig}
\usepackage{fancybox}
\usepackage{bm}
\newtheorem{theorem}{Theorem}[section]
\newtheorem{lemma}[theorem]{Lemma}

\theoremstyle{definition}
\newtheorem{definition}[theorem]{Definition}

\newtheorem{remark}[theorem]{Remark}
\numberwithin{equation}{section}

\numberwithin{equation}{section}

\newsavebox{\savepar}

\begin{document}
	
	\title{An existence result for singular fractional Kirchhoff-Schr\"{o}dinger-Poisson system}
	\author{Sekhar Ghosh\footnote{sekharghosh1234@gmail.com}\\
		\small{Department of Mathematics, National Institute of Technology Rourkela, India}
	}
	\date{}
	\maketitle

	\begin{abstract}
		\noindent In this paper, we study the existence of infinitely many weak solutions to a fractional Kirchhoff-Schr\"{o}dinger-Poisson system involving the weak singularity, i.e. when $0<\gamma<1$. Further, we obtain the existence of a solution with the strong singularity, i.e. when $\gamma>1$. We employ variational techniques to prove the existence and multiplicity results. Moreover, a $L^{\infty}$ estimate is obtained by using the Moser iteration method.\\
		{\bf keywords}: Fractional Sobolev Space, Genus, Symmetric mountain pass theorem, Nehari manifold, Singularity.
	\end{abstract}
	\section{Introduction}
	\noindent In this paper, we study the following fractional Kirchhoff-Schr\"{o}dinger-Poisson system involving a singular term.	
	\begin{align}\label{problem main}
	\begin{split}
	\left(a+b\int_{Q}\frac{(u(x)-u(y))^2}{|x-y|^{N+2s}}\right)(-\Delta)^{s} u+\phi u&=\lambda h(x)u^{-\gamma}+f(x,u)~\text{in}~\Omega,\\
		(-\Delta)^{s}\phi&=u^{2}~\text{in}~\Omega,\\
	u&>0~\text{in}~\Omega,\\
	u&=\phi=0~\text{in}~\mathbb{R}^N\setminus\Omega,
	\end{split}
	\end{align}
	where $\Omega$ is a smooth bounded domain in $\mathbb{R}^{N}$, $a, b\geq 0, a+b>0, \gamma> 0, \lambda>0$, $h\in L^{1}(\Omega)$, $h(x) >0$ a.e. in $\Omega$ and $f$ has some growth conditions.\\
	In the recent time elliptic PDEs involving singularity has drawn interest to many researchers for both the local as well as the nonlocal operators. A noteworthy application of the fractional Laplacian operator can be found in \cite{valdinoci2009long} and the references therein. Further the application can be seen in the field of fluid dynamics, in particular to the study the thin boundary layer properties for viscous fluids \cite{thin bdry}, in probability theory to study the Levy process \cite{bertoin1996levy}, in finance  \cite{cont2004financial}, in free boundary obstacle problems \cite{silvestre2008}. Another application of PDEs involving these type of nonlocal operator is in the field of image processing to find a clear image $u$ from a given noisy image $f$ \cite{bua,kinder}. Readers who are interested to know further details on applications of PDEs involving nonlocal operators can also refer to \cite{servadei2016book,dinezza2012}.\\
	The following model problem of the type \eqref{problem main} was first introduced by Kirchhoff \cite{kirchhoff prob} as a generalization of the D'Alembert wave equation.
	$$\rho\frac{\partial^2u}{\partial t^2}-\left(a+b\int_{0}^{l}\left|\frac{\partial u}{\partial x}\right|^2dx\right)\frac{\partial^2u}{\partial x^2}=g(x,u)$$
	where $a, b, \rho$ are positive constants and $l$ is the changes in the length of the strings due to the vibrations. Recently, Fiscella and Valdinoci (\cite{valdinoci2014}, Appendix A), introduced the fractional Kirchhoff type problem  considering the fractional length of the string for nonlocal measurements. A physical application of nonlocal Kirchhoff type problem can be found in \cite{pucci2016}.\\
	The problem \eqref{problem main} is said to be degenerate if $a=0$ and $b>0$. Otherwise, if both $a>0$ and $b>0$, we say the problem \eqref{problem main} is non degenerate. For $a>0$ and $b=0$, the problem \eqref{problem main} reduces to Schr\"{o}dinger-Poisson system. For $b=\phi=0$, a vast amount of study to prove the existence, multiplicity and regularity of solutions to the problem of type \eqref{problem main} has been done involving both the local operator ($s=1$) as well as the nonlocal operator ($0<s<1$) with a singularity for both $0<\gamma<1$ and $\gamma>1$ and a power nonlinearity or an $L^1$ data or both. The literature is so vast that it is almost impossible to enlist all of them here. A few of such studies can be found in \cite{boccardo2010semilinear,canino2017nonlocal,crandall1977dirichlet,positivity,ghosh2018singular,giacomoni2009multiplicity,lazer1991singular,lei2015,liao2015,liu2013,oliva,saoudi2017critical,ghosh2018multiplicity,sun2013,sun2001,sun2014,yang2003} and the references therein.\\
	For $b= f(x,u)=\phi(x)=0$, the problem \eqref{problem main}, reduces to a purely singular problem. In their celebrated article, Lazer and McKenna \cite{lazer1991singular}, have studied the purely singular problem involving the Laplacian operator, i.e. for $s=1$ and $b=f(x,u)=\phi(x)=0$. The authors in \cite{lazer1991singular}, has proved that the problem has a unique $C^1(\bar{\Omega})$ solution iff $0<\gamma<1$ and it has a $H_0^1(\Omega)$ solution iff $\gamma<3$. Later in \cite{sun2014}, the author proved that if $\gamma\geq3$, then the singular problem can not have $H_0^1(\Omega)$ solution.\\
	Similar to the study with the Laplacian operator, Canino et al. \cite{canino2017nonlocal}, have studied the nonlocal PDE involving singularity. In \cite{canino2017nonlocal}, the authors considered the problem
	\begin{align}\label{canino singular}
	(-\Delta_p)^s u&=\lambda\frac{a(x)}{u^{\gamma}} ~\&~ u>0 ~\text{in}~\Omega,\nonumber\\
	u&=0~\text{in}~\mathbb{R}^N\setminus\Omega.
	\end{align}
	The authors in \cite{canino2017nonlocal}, has guaranteed the existence of unique solution in $W_0^{s,p}(\Omega)$ for $0<\gamma\leq 1$ and in $W_{loc}^{s,p}(\Omega)$ for $\gamma>1.$ For $a(x)\equiv 1=\lambda$, Fang \cite{fang2014existence}, has proved the existence of a unique $C^{2, \alpha}(\Omega)$ solution for $0<\alpha<1$. One of the earliest study to show the existence of multiple solutions was made by Crandall et al. \cite{crandall1977dirichlet} involving the Laplacian operator. Further references on multiplicity involving local operator can be bound in \cite{giacomoni2009multiplicity} and the references therein. Recently, Saoudi et al. in \cite{ghosh2018multiplicity} has guaranteed the existence of at least two solutions by using min-max method with the help of modified Mountain Pass theorem involving fractional $p$-Laplacian operator. Saoudi in \cite{saoudi2017critical}, obtained two solutions involving fractional Laplacian operator For further references on the study of multiple solutions, refer \cite{giacomoni2009multiplicity,ghosh2018multiplicity} and the references therein.\\
	Recently, the study of the Kirchhoff-Schr\"{o}dinger-Poisson system has drawn interest to many researchers. See for instance \cite{fiscella2019,fiscella2019-1,lei2015,fuyi2013,li2017,liao2016,liao2015,liu2013,sun2013,sun2001,zhang2016,jmp} and the references therein for a detailed study of existence, uniqueness and multiplicity of Kirchhoff type problem. Most of these studies, the authors used variational techniques, in particular min-max method, sub-super solution method, Nehari manifold method and Mountain Pass theorem to guarantee the existence and multiplicity of solutions.\\
	In the recent past, for $a=0$ and $0<\gamma<1$, Fiscella \cite{fiscella2019} has obtained two distinct solution involving fractional Laplacian operator by variational technique. Later, for $a>0$, $0<\gamma<1$, Fiscella and Mishra \cite{fiscella2019-1} proved the multiplicity by Nehari manifold method. In \cite{liao2015,liu2013}, the authors studied the multiplicity of solutions involving singular nonlinearity. In \cite{fuyi2013}, Li and Zhang have studied the existence, uniqueness and multiplicity of solution(s) for Schr\"{o}dinger-Poisson system without compactness conditions. On the other hand, Zhang \cite{zhang2016} has studied for a Schr\"{o}dinger-Poisson system. For a detailed study on Schr\"{o}dinger-Poisson system, one can see \cite{antonio2008,antonio2013,ruiz2006} and there references therein. Liao et al. \cite{liao2016} has guaranteed the existence and uniqueness of solution for Kirchhoff type problem involving singularity. The author in \cite{sun2013}, provided a compatibility criterion to obtain the existence of solution for $b=\phi=0, \gamma>1$ and $f(x,t)=t^p, 0<p<1$. The author in \cite{sun2013}, proved that the problem \eqref{problem main} has a  $H_{0}^{1}(\Omega)$ solution if and only if the following compatibility condition for the pair $(h, \gamma)$
	 \begin{equation}\label{compatibility local}
	 \int_{\Omega}h(x)|u_{0}|^{1-\gamma}<\infty~\text{for some}~ u_{0}\in H_{0}^{1}(\Omega)
	 \end{equation} holds true.\\
	  Recently, Zhang \cite{jmp}, proved a necessary and sufficient condition for the existence of solution for a Kirchhoff-Schr\"{o}dinger-Poisson system involving Laplacian operator with strong singularity.\\
	 In their pioneering work, Ambrosetti and Rabinowitz \cite{ambrosetti1973dual} has guaranteed the existence of infinitely many solutions to the problem of the type \eqref{problem main} for $a=s=1$, $b=\lambda=\phi=0$ by introducing the well known (AR) condition on $f$. In fact the (AR) condition has proved to be an important tool to obtain multiplicity of solutions. One can see \cite{binlin2015superlinear,positivity,wang infi} and the references therein for further details on infinitely many solutions. For $b=\phi=\lambda=0$, Binlin et al. \cite{binlin2015superlinear}, has proved the existence of infinitely many solutions for a superlinear data $f$. The study of Kirchhoff type problem to obtain infinitely many solution can be found in \cite{figi2015,li converge,zhao infi} and the references therein. For $\lambda=0$, Li et al. \cite{wang infi} guaranteed the existence of infinitely many solutions to the problem \eqref{problem main} for a sublinear data $f$. Recently, for $0<\gamma<1$, $b=\phi=0$, Ghosh and Choudhuri \cite{positivity}, guaranteed the existence of infinitely many solutions involving the fractional Laplacian operator. In all of these studies referred here that consists of infinitely many solutions, the authors used the symmetric Mountain Pass theorem under the crucial assumption that the data $f$ is odd. The authors in \cite{positivity}, assumed the following growth conditions on $f$.
	 	\begin{itemize}
	 	\item[(A1)] $f\in C(\Omega\times\mathbb{R}, \mathbb{R})$ and $\exists$ $\delta>0$ such that $\forall\,x\in\Omega$ and $|t|\leq\delta,$  $f(x,-t)=-f(x,t).$
	 	\item[(A2)] $\lim\limits_{t\rightarrow0}\frac{f(x,t)}{t}=+\infty$ uniformly on $\Omega.$
	 	\item[(A3)] There exists $r>0$ and $p\in(1-\gamma, 2)$ such that $\forall,\,x\in\Omega$ and $|t|\leq r$, $tf(x,t)\leq pF(x,t)$, where $F(x,t)=\int_{0}^{t}f(x,\tau) d\tau.$
	 \end{itemize}
 	Motivated from \cite{positivity,wang infi,sun2013}, we consider the fractional Kirchhoff-Schr\"{o}dinger-Poisson system \eqref{problem main} involving singularity. To the best of the author's knowledge, there in no study of infinitely many solutions for a fractional Kirchhoff-Schr\"{o}dinger-Poisson system involving a power nonlinearity and a singularity ($0<\gamma<1$) in the literature. On the other hand, the study of the problem \eqref{problem main} with strong singularity ($\gamma>1$) is comparatively challenging and hence can also be seen as a new addition to the literature involving nonlocal operator. One can expect the compatibility condition \eqref{compatibility local} as in \cite{sun2013}, for the existence of $X_0$ solution(s) to be 
 	\begin{equation}\label{compatibility}
 	\int_{\Omega}h(x)|u_{0}|^{1-\gamma}<\infty~\text{for some}~ u_{0}\in X_0.
 	\end{equation} 	
 	The following two Theorems are the main results proved in this article.
	\begin{theorem}\label{main thm1}
		Assume $a, b\geq0, a+b>0$, $h\in L^{1}(\Omega), h>0$ a.e. in $\Omega$ and ($A1$)-($A3$) holds. Then for $0<\gamma<1$ and for any $\lambda\in(0,\Lambda)$, the problem \eqref{problem main} has a sequence of positive weak solutions ${u_n}\subset X_0\cap L^{\infty}(\Omega)$ such that $I(u_n) < 0$, $I(u_n)\rightarrow 0^{-}$ and $u_n \rightarrow 0$ in $X_0$. (See Section 2. for notations).
	\end{theorem}
 \begin{remark}
	Note that in Theorem \ref{main thm1} there is no any restriction condition for $f$ in $t$ at infinity.
\end{remark}
	\begin{theorem}\label{main theorem}
		Assume $a, b\geq0, a+b>0$, $h\in L^{1}(\Omega), h>0$ a.e. in $\Omega$ and $f(x,u)=k(x)u^p$ such that $k\in L^{\infty}(\Omega)$ with $k>0$, $0<p<1$. Then, for $\gamma>1$ and for any $\lambda\in(0,\Lambda)$, the problem \eqref{problem main} has a weak solution in $X_0$ if and only if \eqref{compatibility} holds true.
	\end{theorem}
	\begin{remark}
		If we assume $f(x,\cdot)\equiv0$, then the problem \eqref{problem main} possesses a unique solution.
	\end{remark}
	\noindent The paper is organised as follows. In Section 2, we will first give some mathematical formulation and define the space $X_0$. Moreover, we will discuss some preliminary properties of $\phi$ and prove that $\Lambda$ has a finite range. In the subsequent sections, Section 3 and Section 4, we will obtain the results as stated in Theorem \ref{main thm1} and Theorem \ref{main theorem} respectively.
	
	\section{Mathematical formulations}
	This section is devoted to give a few important results of fractional Sobolev spaces, embeddings, variational formulations and space setup. Let $\Omega$ be open bounded domain of $\mathbb{R}^N$ and $Q=\mathbb{R}^{2N}\setminus((\mathbb{R}^{N}\setminus\Omega)\times(\mathbb{R}^{N}\setminus\Omega))$. For $0<s<1$, the space ($X,\|.\|$), which is an intermediary Banach space between $H^1(\Omega)$ and $L^2(\Omega)$, is defined as
	 \begin{eqnarray}
	X&=&\left\{u:\mathbb{R}^N\rightarrow\mathbb{R}~\text{is measurable}, u|_{\Omega}\in L^2(\Omega) ~\text{and}~\frac{|u(x)-u(y)|}{|x-y|^{\frac{N+2s}{2}}}\in L^{2}(Q)\right\}\nonumber
	\end{eqnarray}
	equipped with the norm 
	\begin{eqnarray}
	\|u\|_X&=&\|u\|_{2}+[u]_2,\nonumber
	\end{eqnarray}
where $[u]_2=\left(\int_{Q}\frac{|u(x)-u(y)|^2}{|x-y|^{N+2s}}dxdy\right)^{\frac{1}{2}}$ refers to the Gagliardo semi norm. Due to the zero Dirichlet boundary condition, it is natural to consider the space
	$$X_0=\{u\in X:u=0 ~\text{a.e. in}~\mathbb{R}^N \setminus\Omega\},$$ endowed with the following Gagliardo norm on it. 
	\begin{eqnarray}
	\|u\|&=&\left(\int_{Q}\frac{|u(x)-u(y)|^2}{|x-y|^{N+2s}}dxdy\right)^{\frac{1}{2}}.\nonumber
	\end{eqnarray}
	The space $(X_0, \|.\|)$ is a Hilbert space \cite{servadei2012mountain}. The best Sobolev constant is defined as 
	\begin{equation}\label{sobolev const}
	S=\underset{u\in X_0\setminus\{0\}}{\inf}\cfrac{\int_{Q}\frac{|u(x)-u(y)|^2}{|x-y|^{N+2s}}dxdy}{\left(\int_\Omega|u|^{2_s^*}dx\right)^{\frac{2}{2_s^*}}}
	\end{equation}
	Let $\Omega$ be a bounded Lipschitz domain in $\mathbb{R}^N$. Then for every $q\in[1, 2_s^*]$, the space $X_0$ is continuously embedded in $L^q(\Omega)$ and for every $q\in[1, 2_s^*)$, the space $X_0$ is compactly embedded in $L^q(\Omega)$, where $2_s^*=\frac{2N}{N-2s}.$ Prior to define the weak solution to our problem, let us first consider the following problem
	\begin{align}\label{poisson eqn}
	(-\Delta)^{s}\phi=u^{2}~\text{in}~\Omega,\nonumber\\
	\phi=0~\text{in}~\mathbb{R}^N\setminus\Omega.
	\end{align}
	In light of the Lax-Milgram theorem, for every $u\in  X_0$, the problem \eqref{poisson eqn} has a unique solution $\phi_{u}\in  X_0$ and we have the following Lemma consisting some properties of the solution $\phi_{u}$.
	\begin{lemma}\label{phi prop}
		For each solution $\phi_u\in X_0$ of \eqref{poisson eqn}, we have
		\begin{enumerate}[label=(\roman*)]
			\item $\|\phi_{u}\|^{2}=\int_{\Omega}\phi_{u}u^{2}dx=\int_{\Omega}|(-\Delta)^{s/2}\phi_{u}|^{2}dx\leq C_{\phi}\|u\|^{4}, ~\forall~u\in X_0$;
			\item $\phi_{u}\geq 0$. Moreover, $\phi_{u}>0$ if $u\neq 0$;
			\item for all $t\neq 0$, $\phi_{tu}=t^{2}\phi_{u}$;
			\item $\|u_n-u\|\rightarrow 0$ implies that $\|\phi_{u_n}-\phi_u\|\rightarrow 0$ and $\int_{\Omega}\phi_{u_{n}}u_{n}vdx\rightarrow\int_{\Omega}\phi_{u}uvdx$, for any $v\in X_0$;
			\item for any $u, v\in X_0$, we have $\int_{\Omega}(\phi_{u}u-\phi_{v}v)(u-v)dx\geq\frac{1}{2}\|\phi_{u}-\phi_{v}\|^{2}.$
		\end{enumerate}	
	\end{lemma}
	\noindent Now by replacing $\phi_u$ in place of $\phi$ in \eqref{problem main}, the problem \eqref{problem main} reduces to the following Dirichlet boundary value problem
	\begin{align}\label{problem reduced}
	\begin{split}
	\left(a+b\int_{Q}\frac{(u(x)-u(y))^2}{|x-y|^{N+2s}}\right)(-\Delta)^{s} u+\phi_u u&=\lambda h(x)u^{-\gamma}+f(x,u)~\text{in}~\Omega,\\
	u&>0~\text{in}~\Omega,\\
	u&=0~\text{in}~\mathbb{R}^N\setminus\Omega,
	\end{split}
	\end{align}
	We now define a weak solution to the problem \eqref{problem reduced}.
	\begin{definition}\label{weak solution defn}
		A function $u\in X_0$ is a weak solution to the problem \eqref{problem reduced}, if $u>0$ and
	{\small \begin{align}\label{weak formulation}
		&(a+b[u]^{2})\int_{Q}(-\Delta)^{s/2} u\cdot(-\Delta)^{s/2}\psi+\int_{\Omega}\phi_{u}u\psi -\lambda\int_{\Omega}h(x)u^{-\gamma}\psi -\int_{\Omega}f(x,u)\psi=0, 
		\end{align}}
	for every $\psi\in X_0.$
	\end{definition}
	\noindent The associated energy functional to the problem \eqref{problem reduced} is defined as
	{\small \begin{align}\label{functional}
	I(u)=\frac{a}{2}\|u\|^{2}+\frac{b}{4}\|u\|^{4}&+\frac{1}{4}\int_{\Omega}\phi_{u}u^{2}-\frac{\lambda}{1-\gamma}\int_{\Omega}h(x)|u|^{1-\gamma}-\int_{\Omega}F(x,u),~ u\in X_0,
	\end{align}}
	\noindent where $F(x,u)=\int_{0}^{u}f(x,t)dt$. Observe that for $0<\gamma<1$, the term $\int_{\Omega}h(x)|u|^{1-\gamma}<\infty$ but the functional $I$ fails to be $C^1$. Therefore, by modifying the problem \eqref{problem main}, we will use the Kajikiya's Symmetric mountain pass theorem \cite{kajikiya2005critical} and a cut-off technique developed in \cite{clark1972variant} to obtain a $C^1$ functional to guarantee the existence of infinitely many solutions. On the other hand, for $\gamma>1$ the integral $\int_{\Omega}h(x)|u|^{1-\gamma}dx$ is not finite for $u \in X_0$. Therefore, the energy functional $I$ fails to be continuous and we cannot use the usual variational technique to guarantee the existence of solution. We will use arguments from \cite{sun2013} to obtain a weak solution to the problem \eqref{problem reduced}. Similar type of results can also be found in \cite{zhang2016}. We now state and prove the following Lemma to guarantee a finite range for $\Lambda$,which is defined as
		$$\Lambda=\inf\{\lambda>0: ~\text{The problem \eqref{problem main} has no solution}\}.$$
	\begin{lemma}\label{lambda finite}
		Assume $a, b, \gamma>0$, ($A1$)-($A3$) and \eqref{compatibility} holds. Then $0\leq\Lambda<\infty$.
	\end{lemma}
	\begin{proof}
		By definition, $\Lambda\geq 0$. Let $\phi_1>0$ be the first eigenfunction \cite{brasco2016second} corresponding to the first eigenvalue $\lambda_1$ for the fractional Laplacian operator. Then we have
		\begin{align}
		\begin{split}
		(-\Delta)^s\phi_1&=\lambda_1\phi_1~\text{in}~\Omega\\
		\phi_1&>0~\text{in}~\Omega\\
		\phi_1&>0~\text{in}~\mathbb{R}^N\setminus\Omega.
		\end{split}
		\end{align}
		Therefore, by putting $\phi_1$ as the test function in Definition \ref{weak solution defn}, we obtain
		\begin{align}\label{lambda invalid}
		\begin{split}
		\lambda_1\int_{\Omega}(a+b\|u\|^2)u\phi_1dx&=\int_{\Omega}(a+b\|u\|^2)(-\Delta)^s\phi_1udx\\
		&=\int_{\Omega}\left(\lambda h(x)u^{-\gamma}+f(x,u)-\phi_uu \right)\phi_1dx
		\end{split}
		\end{align}
		At this stage, we choose $\tilde{\Lambda}>0$ such that
		$$\tilde{\Lambda}h(x_0)t^{-\gamma}+f(x_0,t)>2\lambda_1t(a+bt^2)+\phi_tt$$ for all $t>0$ and for some $x_0\in\Omega$, which gives a contradiction to \eqref{lambda invalid}. Hence $\Lambda<\infty$.
	\end{proof}
	\noindent In the subsequent two sections, we establish the existence of solution(s).	
	\section{Existence of infinitely many solutions for $0<\gamma<1$.}	
	\noindent We begin this section with the definition of genus of a set.
	\begin{definition}\label{genus} {(\bf{Genus})}
		Let $X$ be a Banach space and $A\subset X$. A set $A$ is said to be symmetric if $u\in A$ implies $(-u)\in A$. Let $A$ be a closed, symmetric subset of $X$ such that $0\notin A$. We define a genus $\gamma(A)$ of $A$ by the smallest integer $k$ such that there exists an odd continuous mapping from $A$ to $\mathbb{R}^{k}\setminus\{0\}$. We define $\gamma(A)=\infty$, if no such $k$ exists.
	\end{definition}
	\noindent We now define the following family of sets, $$\Gamma_n=\{A_n\subset X: A_n~\text{is closed, symmetric and}~ 0\notin A_n~\text{such that the genus}~ \gamma(A_n)\geq n\}.$$
	Further, we will use the following version of the symmetric Mountain Pass Theorem from \cite{kajikiya2005critical}.
	\begin{theorem}\label{sym mountain}
		Let $X$ be an infinite dimensional Banach space and $\tilde{I}\in C^1(X,\mathbb{R})$ satisfies the following
		\begin{itemize}
			\item[(i)] $\tilde{I}$ is even, bounded below, $\tilde{I}(0)=0$ and $\tilde{I}$ satifies the $(PS)_c$ condition.
			\item[(ii)] For each $n\in\mathbb{N}$, there exists an $A_n\in\Gamma_n$ such that $\sup\limits_{u\in A_n}\tilde{I}(u)<0.$ 
		\end{itemize}
		Then for each $n\in\mathbb{N}$, $c_n=\inf\limits_{A\in \Gamma_n}\sup\limits_{u\in A}\tilde{I}(u)<0$ is a critical value of $\tilde{I}.$
	\end{theorem}
	
	\noindent We will modify the problem \eqref{problem reduced} to apply the symmetric Mountain Pass Theorem as follow
	\begin{align}\label{main2}
\begin{split}
\left(a+b\int_{Q}\frac{(u(x)-u(y))^2}{|x-y|^{N+2s}}\right)(-\Delta)^{s} u+\phi_u u&=\lambda h(x)sign(u)|u|^{-\gamma}+f(x,u)~\text{in}~\Omega,\\
u&=0~\text{in}~\mathbb{R}^N\setminus\Omega,
\end{split}
	\end{align}
	
	\noindent The associated energy functional to the problem \eqref{main2} is defined as
	\begin{align}\label{energy modified}
	J(u)=\frac{a}{2}\|u\|^{2}+\frac{b}{4}\|u\|^{4}&+\frac{1}{4}\int_{\Omega}\phi_{u}u^{2}-\frac{\lambda}{1-\gamma}\int_{\Omega}h(x)|u|^{1-\gamma}-\int_{\Omega}F(x,u),~ u\in X_0,
	\end{align}
	where $F(x,u)=\int_{0}^{|u|}f(x,t)dt$. Observe that the functional $J$ is even by using the assumption ($A1$) and Lemma \ref{phi prop}(iii). We now define a weak solution to the modified problem \eqref{main2}.
	\begin{definition}\label{weak modified}
		A function $u\in X_0$ is a weak solution of \eqref{main2}, if $\phi |u|^{-\gamma}\in L^1(\Omega)$ and
	{\small \begin{align}
		&(a+b[u]^{2})\int_{Q}(-\Delta)^{s/2} u\cdot(-\Delta)^{s/2}\psi+\int_{\Omega}\phi_{u}u\psi -\int_{\Omega}\left(\lambda h(x)sign(u)|u|^{-\gamma}+f(x,u)\right)\psi=0, 
		\end{align}}
	\end{definition}
	\noindent for every $\psi\in X_0.$ Observe that if $u>0$ a.e. in $\Omega$, then weak solutions to the problem \eqref{main2} and to the problem \eqref{problem reduced} coincide. Therefore, it is sufficient to obtain a sequence of nonnegative weak solutions to the problem \eqref{problem reduced}. We now extend and modify $f(x,u)$ for $u$ outside a neighbourhood of $0$ by $\tilde{f}(x, u)$ as follow. We will follow \cite{clark1972variant} by considering a cut-off problem. Choose $l>0$ sufficiently small such that $0<l\leq\frac{1}{2}\min\{\delta, r\}$, where $\delta$ and $r$ are same as in the assumptions on $f$. We now define a $C^1$ function $\xi:\mathbb{R}\rightarrow\mathbb{R}^+$ such that $0\leq\xi(t)\leq1$ and 
	$$\xi(t)=\begin{cases}
	1, ~\text{if}~ |t|\leq l\\
	\xi ~\text{is decreassing, if}~ l\leq t\leq 2l\\
	0,~\text{if}~ |t|\geq 2l.
	\end{cases}$$
	We now consider the following cut-off problem by defining $\tilde{f}(x, u)=f(x, u)\xi(u)$.
	
	\begin{align}\label{main3}
\begin{split}
\left(a+b\int_{Q}\frac{(u(x)-u(y))^2}{|x-y|^{N+2s}}\right)(-\Delta)^{s} u+\phi_u u&=\lambda h(x)sign(u)|u|^{-\gamma}+\tilde{f}(x,u)~\text{in}~\Omega,\\
u&=0~\text{in}~\mathbb{R}^N\setminus\Omega,
\end{split}
	\end{align}
	\noindent The associated energy functional to the problem \eqref{main3} is defined as
	\begin{align}\label{energy cutoff}
	\tilde{I}(u)=\frac{a}{2}\|u\|^{2}+\frac{b}{4}\|u\|^{4}&+\frac{1}{4}\int_{\Omega}\phi_{u}u^{2}-\frac{\lambda}{1-\gamma}\int_{\Omega}h(x)|u|^{1-\gamma}-\int_{\Omega}\tilde{F}(x, u)dx, ~ u\in X_0.
	\end{align}	
	\noindent We define a weak solution to the problem \eqref{main3} as follows.
	\begin{definition}\label{weak cutoff}
		A function $u\in X_0$ is a weak solution of \eqref{main3}, if $\phi |u|^{-\gamma}\in L^1(\Omega)$ and
	{\small \begin{align}
		&(a+b[u]^{2})\int_{Q}(-\Delta)^{s/2} u\cdot(-\Delta)^{s/2}\psi+\int_{\Omega}\phi_{u}u\psi -\int_{\Omega}(\lambda h(x)sign(u)|u|^{-\gamma}+\tilde{f}(x,u))\psi=0
		\end{align}}
	\end{definition}
	\noindent for every $\psi\in X_0$. Again, if  $\|u\|_{\infty}\leq l$ holds, then the weak solutions of \eqref{main3} and the weak solutions of \eqref{main2} coincide. We establish the existence result for the problem \eqref{main3}. Finally, we prove our main theorem by showing that the solutions to \eqref{main3} are positive and  $\|u\|_{\infty}\leq l$.

	\noindent We first prove the following Lemmas which are the hypotheses to the Symmetric mountain pass theorem.
	\begin{lemma}\label{lemma ps}
		The functional $\tilde{I}$ is bounded from below and satisfies $(PS)_c$ condition.
	\end{lemma}
	\begin{proof}
		By the definition of $\xi$ and using the H\"{o}lder's inequality, we get 
		\begin{align*}
		\tilde{I}(u)&\geq\frac{a}{2}\|u\|^{2}+\frac{b}{4}\|u\|^{4}+\frac{1}{4}\int_{\Omega}\phi_{u}u^{2}-C\|u\|^{1-\gamma}-C_1\\
		&\geq\frac{a}{2}\|u\|^{2}+\frac{b}{4}\|u\|^{4}- C\|u\|^{1-\gamma}-C_1
		\end{align*}
		where, $C$, $C_1$ are nonnegative constants. Since $a, b>0$, this implies that $\tilde{I}$ is coercive and bounded from below in $X_0$. Let $\{u_n\}\subset X_0$ be a Palais Smale sequence for the functional $\tilde{I}$. Therefore, by using the coerciveness property of  $\tilde{I}$ we have $\{u_n\}$ is bounded in $X_0$. Thus, we may assume that $\{u_n\}$ has a subsequence (still denoted by $\{u_n\}$) such that $u_n\rightharpoonup u$ in $X_0$. Therefore, we have
		\begin{equation}\label{convergence weak}
		\int_{Q}(-\Delta)^{s/2} u_n\cdot(-\Delta)^{s/2}\psi dxdy\longrightarrow\int_{Q}(-\Delta)^{s/2} u\cdot(-\Delta)^{s/2}\psi dxdy
		\end{equation} 
		for all $\phi\in X_0.$
		By the embedding result \cite{servadei2012mountain}, we can assume for every $q\in[1, 2_s^*)$
		\begin{align}
		u_n&\longrightarrow u ~\text{in}~ L^q(\Omega),\label{embed strong}\\
		u_n(x)&\longrightarrow u(x) ~\text{a.e.}~ L^q(\Omega).\label{embed pointwise}
		\end{align}
		Therefore, from Lemma A.1 \cite{willem1997minimax}, we get that there exists $g\in L^q(\Omega)$ such that
		\begin{equation}\label{appendeix A1}
		|u_n(x)|\leq g(x) ~\text{a.e. in}~ \Omega, \forall\,n\in\mathbb{N}.
		\end{equation}
		Now on using \eqref{embed strong}, \eqref{embed pointwise}, \eqref{appendeix A1} and applying the Lebesgue dominated convergence theorem, we obtain
		\begin{equation}\label{convergence f tilla}
		\int_{\Omega}\tilde{f}(x,u_n)udx\rightarrow\int_{\Omega}\tilde{f}(x,u)udx ~\text{and}~ \int_{\Omega}\tilde{f}(x,u_n)u_ndx\rightarrow\int_{\Omega}\tilde{f}(x,u)udx.
		\end{equation}
		Moreover, 
		\begin{equation}\label{convergence phi}
		\int_{\Omega}\phi_{u_n}u_nudx\rightarrow\int_{\Omega}\phi_{u}u^2dx ~\text{and}~ \int_{\Omega}\phi_{u_n}u_n^2dx\rightarrow\int_{\Omega}\phi_{u}u^2dx.
		\end{equation}		
		Again, on using the H\"{o}lder's inequality and passing the limit $n\rightarrow\infty$, we get
		\begin{align}
		\begin{split}
		\int_{\Omega}u_n^{1-\gamma}dx&\leq\int_{\Omega}u^{1-\gamma}dx+\int_{\Omega}|u_n-u|^{1-\gamma}dx\\
		&\leq\int_{\Omega}u^{1-\gamma}dx+C\|u_n-u\|_{L^2(\Omega)}^{1-\gamma}\\
		&=\int_{\Omega}u^{1-\gamma}dx +o(1).
		\end{split}
		\end{align}
		Similarly, we have
		\begin{align}
		\begin{split}
		\int_{\Omega}u^{1-\gamma}dx&\leq\int_{\Omega}u_n^{1-\gamma}dx+\int_{\Omega}|u_n-u|^{1-\gamma}dx\\
		&\leq\int_{\Omega}u_n^{1-\gamma}dx+C\|u_n-u\|_{L^2(\Omega)}^{1-\gamma}\\
		&=\int_{\Omega}u_n^{1-\gamma}dx +o(1).
		\end{split}
		\end{align}
		Therefore,
		\begin{equation}\label{convergence singular}
		\int_{\Omega}u_n^{1-\gamma}dx=\int_{\Omega}u^{1-\gamma}dx+o(1).
		\end{equation}
		Since, $\{u_n\}$ is a Palais Smale sequence of $\tilde{I}$ therefore, by weak convergence, we gave
		\begin{equation}\label{ps seq}
		\langle\tilde{I}'(u_n)-\tilde{I}'(u), u_n-u\rangle=o(1)~\text{as}~n\rightarrow\infty.
		\end{equation}
		On the other hand,
		\begin{align}
		\langle\tilde{I}'(u_n)-\tilde{I}'(u), (u_n-u)\rangle&= (a+b[u_n]^{2})\langle u_n, (u_n-u)\rangle-(a+b[u]^{2})\langle u, (u_n-u)\rangle\nonumber\\
		&+\int_{\Omega}[(\phi_{u_n}u_n-\phi_{u}u) -\lambda h(x)(sign(u_n)|u_n|^{-\gamma}- sign(u)|u|^{-\gamma})](u_n-u)\nonumber\\
		&-\int_{\Omega}(\tilde{f}(x,u_n)-\tilde{f}(x,u))(u_n-u)
		\end{align}
		Now, on using \eqref{convergence f tilla}, \eqref{convergence phi} and \eqref{convergence singular} we get
		\begin{align}\label{ps seq1}
		&\langle\tilde{I}'(u_n)-\tilde{I}'(u), (u_n-u)\rangle= (a+b[u_n]^{2})\langle u_n, (u_n-u)\rangle-(a+b[u]^{2})\langle u, (u_n-u)\rangle+o(1)
		\end{align}
		as $n\rightarrow\infty$. Observe that
		\begin{align}\label{ps seq2}
		&(a+b[u_n]^{2})\langle u_n, (u_n-u)\rangle-(a+b[u]^{2})\langle u, (u_n-u)\rangle\nonumber\\
		&=(a+b[u_n]^{2})[u_n-u]^2+b([u_n]^{2}-[u]^{2})\langle u, (u_n-u)\rangle.
		\end{align}
		Since, the sequence $(a+b[u_n]^{2})$ is bounded in $X_0$. Thus by using the definition of weak convergence, we get
		\begin{align}\label{ps seq3}
		b([u_n]^{2}-[u]^{2})\langle u, (u_n-u)\rangle=o(1)~\text{as}~n\rightarrow\infty.
		\end{align}
		Therefore, from \eqref{ps seq2} and \eqref{ps seq3}, we obtain
		\begin{align}\label{ps seq4}
		&(a+b[u_n]^{2})\langle u_n, (u_n-u)\rangle-(a+b[u]^{2})\langle u, (u_n-u)\rangle\geq a[u_n-u]^2~\text{as}~n\rightarrow\infty.
		\end{align}
		Finally, on using \eqref{ps seq}, \eqref{ps seq1} and \eqref{ps seq4}, we conclude that
		\begin{align}
		o(1)\geq\min\{a,1\}\|u_n-u\|^2+o(1)~\text{as}~n\rightarrow\infty.
		\end{align}
		Hence, $u_n\rightarrow u$ strongly in $X_0$ and this completes the proof.
	\end{proof}
	
	\begin{lemma}\label{lemma genus}
		For any $n\in\mathbb{N}$, there exists a closed, symmetric subset $A_n\subset X_0$ with $0\notin A_n$ such that the genus $\gamma(A_n)\geq n$ and $\sup\limits_{u\in A_n}\tilde{I}(u)<0.$
	\end{lemma}
	\begin{proof}
		We will first obtain the existence of a closed, symmetric subset $A_n$ of $X_0$ over every finite dimensional subspace such that $\gamma(A_n)\geq n.$ Let $X_k$ be a subspace of $X_0$ such that $\dim (X_k)=k.$ Since, every norm over a finite dimensional Banach space are equivalent then there exists a positive constant $M=M(k)$ such that $\|u\|\leq M\|u\|_{L^2(\Omega)}$ for all $u\in X_k.$\\
		{\bf Claim:}
		There exists a positive constant $R$ such that 
		\begin{equation}\label{claim genus}
		\frac{1}{2}\int_{\Omega}|u|^2dx\geq\int_{\{|u|> l\}}|u|^2dx,~\forall\,u\in X_k ~\text{such that}~ \|u\|\leq R.
		\end{equation}
		We proof it by contradiction. Let $\{u_n\}$ be a sequence in $X_k\setminus\{0\}$ such that $u_n\rightarrow0$ in $X_0$ and
		\begin{equation}\label{eq 3.24}
		\frac{1}{2}\int_{\Omega}|u_n|^2dx<\int_{\{|u_n|> l\}}|u_n|^2dx.
		\end{equation}
		Choose, $v_n=\frac{u_n}{\|u_n\|_{L^2(\Omega)}}.$ Then \eqref{eq 3.24} reduces to
		\begin{equation}\label{claim contra eqn}
		\frac{1}{2}<\int_{\{|u_n|> l\}}|v_n|^2dx.
		\end{equation}
		Since, $X_k$ is finite dimensional and $\{v_n\}$ is bounded, we can assume $v_n\rightarrow v$ in $X_0$ upto a subsequence. Therefore, $v_n\rightarrow v$ also in $L^2(\Omega).$ Further observe that, $$m\{x\in\Omega: |u_n|>l\}\rightarrow0 ~\text{as}~ n\rightarrow\infty,$$ since $u_n\rightarrow0$ in $X_0$, where $m$ refers to the Lebesgue measure. This is a contradiction to the equation \eqref{claim contra eqn}. Hence, the claim is established. Again, from the assumption $(A2)$, one can choose $0<l\leq1$ sufficiently small such that, $$\tilde{F}(x,t)=F(x,t)\geq4\left(\frac{a}{2}+\frac{b}{4}+C_{\phi}\right)M^2t^2, ~\forall\, (x,t)\in\Omega\times[0,l].$$
		Hence, for all $u\in X_k\setminus\{0\}$ such that $\|u\|\leq R$ and by using \eqref{claim genus}, we get
		\begin{align*}
		\tilde{I}(u)&\leq\frac{a}{2}\|u\|^2+\frac{b}{4}\|u\|^4+\int_{\Omega}\phi_uu^2dx-\frac{\lambda}{1-\gamma}\int_{\Omega}|h(x)||u|^{1-\gamma}dx-\int_{\{|u|\leq l\}}\tilde{F}(x, u)dx\\
		&\leq\frac{a}{2}\|u\|^2+\frac{b}{4}\|u\|^4+C_{\phi}\|u\|^4-\frac{\lambda}{1-\gamma}\int_{\Omega}|h(x)||u|^{1-\gamma}dx-4\left(\frac{a}{2}+\frac{b}{4}+C_{\phi}\right)M^2\int_{\{|u|\leq l\}}|u|^2dx\\
		&\leq\left(\frac{a}{2}+\frac{b}{4}+C_{\phi}\right)\|u\|^2-\frac{\lambda}{1-\gamma}\int_{\Omega}|h(x)||u|^{1-\gamma}dx\\
		&\hspace{4.2cm}-4\left(\frac{a}{2}+\frac{b}{4}+C_{\phi}\right)M^2\left(\int_{\Omega}|u|^2dx-\int_{\{|u|> l\}}|u|^2dx\right)
		\end{align*}
		\begin{align*}
		&\leq\left(\frac{a}{2}+\frac{b}{4}+C_{\phi}\right)\|u\|^2-\frac{\lambda}{1-\gamma}\int_{\Omega}|h(x)||u|^{1-\gamma}dx-2\left(\frac{a}{2}+\frac{b}{4}+C_{\phi}\right)M^2\int_{\Omega}|u|^2dx\\
		&\leq-\left(\frac{a}{2}+\frac{b}{4}+C_{\phi}\right)\|u\|^2-\frac{\lambda}{1-\gamma}\int_{\Omega}|h(x)||u|^{1-\gamma}dx\\
		&<0,~\text{for all}~u\in X_0~\text{such that}~\|u\|\leq\min\{1, R\}.
		\end{align*}
		\noindent We now choose, $0<\rho\leq\min\{1,R\}$ and $A_n=\{u\in X_n: \|u\|=\rho\}$. Thus $\Gamma_n\neq\phi$. This concludes that $A_n$ is symmetric, closed with $\gamma(A_n)\geq n$ such that $\sup\limits_{u\in A_n}\tilde{I}(u)<0.$
	\end{proof}
	\noindent We now state the following Lemmas which are essential to prove the boundedness of the solutions to the problem \eqref{main3}. The Lemma \ref{bounded l1} and Lemma \ref{bounded l2} are taken from \cite{brasco2016second} and a simple proof can be found in \cite{positivity}.	
	\begin{lemma}\label{bounded l1}
		Let $g:\mathbb{R}\rightarrow\mathbb{R}$ be a convex $C^1$ function. Then for every $c, d, C, D\in\mathbb{R}$ with $C, D>0$ the following inequality holds.
		\begin{equation}
		(g(c)-g(d))(C-D)\leq (c-d)(Cg'(c)-Dg'(d))
		\end{equation}
	\end{lemma}
	\begin{lemma}\label{bounded l2}
		Let $\tilde{h}:\mathbb{R}\rightarrow\mathbb{R}$ be an increasing function, then for $c, d, \tau\in\mathbb{R}$ with $\tau\geq 0$ we have
		\begin{equation}
		[\tilde{H}(c)-\tilde{H}(d)]^2\leq (c-d)(\tilde{h}(c)-\tilde{h}(d))
		\end{equation}
		where, $\tilde{H}(t)=\int_0^t \sqrt{\tilde{h}'(\tau)}d\tau$, for $t\in\mathbb{R}.$
	\end{lemma}
	\noindent The following Lemma is based on the Moser iteration technique, which gives an uniform $L^{\infty}$ bound to the weak solutions of the problem \eqref{main3}.
	\begin{lemma}\label{bounded}
		Let $u\in X_0$ be a positive weak solution to the problem in \eqref{main3}, then $u\in L^{\infty}(\Omega).$
	\end{lemma}
	\begin{proof}
		The proof is based on arguments as in \cite{positivity}. We will make use of the fact that
		$$\int_{Q}\frac{(u(x)-u(y))(\psi(x)-\psi(y))}{|x-y|^{N+2s}}dxdy=C\int_{Q}(-\Delta)^{s/2} u\cdot(-\Delta)^{s/2}\psi dxdy,$$ for $\psi\in X_0.$ For every small $\epsilon>0,$ consider the smooth function  
		\begin{equation*}
		g_{\epsilon}(t)=(\epsilon^2+t^2)^{\frac{1}{2}}
		\end{equation*}
		Note that the function $g_{\epsilon}$ is convex as well as Lipschitz. We choose $\psi=\tilde{\psi} g'_{\epsilon}(u)$ as the test function in \eqref{main3} for all positive $\tilde{\psi}\in C_c^{\infty}(\Omega)$. Now by taking $c=u(x), d=u(y), C=\psi(x)$ and $D=\psi(y)$ in Lemma \ref{bounded l1}, we get
		\begin{align}\label{bound est 1}
		(a+b\|u\|^2)&\int_{Q}\cfrac{(g_{\epsilon}(u(x))-g_{\epsilon}(u(y)))(\tilde{\psi}(x)-\tilde{\psi}(y))}{|x-y|^{N+2s}}dxdy\nonumber\\
		&\leq\int_\Omega\left(|\lambda h(x)u^{-\gamma}+\tilde{f}(x, u)|-\phi_uu\right)|g'_{\epsilon}(u)|\tilde{\psi} dx\nonumber\\
		&\leq\int_\Omega\left(|\lambda h(x) u^{-\gamma}+\tilde{f}(x, u)|\right)|g'_{\epsilon}(u)|\tilde{\psi} dx
		\end{align}
		Since, $g_{\epsilon}(t)\rightarrow|t|$ as $t\rightarrow0$, hence $|g'_{\epsilon}(t)|\leq1$ for all $t\geq0$. Therefore, on using the Fatou's Lemma and passing the limit $\epsilon\rightarrow0$ in \eqref{bound est 1}, we obtain
		\begin{align}\label{bound est 2}
		(a+b\|u\|^2)\int_{Q}\cfrac{(|u(x)|-|u(y)|)(\tilde{\psi}(x)
			-\tilde{\psi}(y))}{|x-y|^{N+2s}}dxdy\leq\int_\Omega\left(|\lambda h(x) u^{-\gamma}+\tilde{f}(x, u)|\right)\tilde{\psi} dx
		\end{align}
		for all $\tilde{\psi}\in C_c^{\infty}(\Omega)$ with $\tilde{\psi}>0.$ The inequality \eqref{bound est 2} remains true for all $\tilde{\psi}\in X_0$ with $\tilde{\psi}\geq0.$ We define the {cut-off} function $u_k=\min\{(u-1)^+, k\}\in X_0$ for $k>0.$ Now for any given $\beta>0$ and $\delta>0$, we choose $\tilde{\psi}=(u_k+\delta)^{\beta}-\delta^{\beta}$ as the test function in \eqref{bound est 2} and get
		\begin{align}\label{bound est 3}
		(a+b\|u\|^2)&\int_{Q}\cfrac{(|u(x)|-|u(y)|)((u_k(x)+\delta)^{\beta}-(u_k(y)+\delta)^{\beta})}{|x-y|^{N+2s}}dxdy\nonumber\\
		&\leq\int_\Omega\left(|\lambda h(x) u^{-\gamma}+\tilde{f}(x, u)|\right)\left((u_k+\delta)^{\beta}-\delta^{\beta}\right) dx
		\end{align}
		Now applying the Lemma \ref{bounded l2} to the function $\tilde{h}(u)=(u_k+\delta)^{\beta},$ we get
		\begin{align}\label{bound est 4}
		\begin{split}
		&(a+b\|u\|^2)\int_{Q}\cfrac{|((u_k(x)+\delta)^{\frac{\beta+1}{2}}
			-(u_k(y)+\delta)^{\frac{\beta+1}{2}})|^2}{|x-y|^{N+2s}}dxdy\\
		&\leq\frac{(\beta+1)^2}{4\beta}(a+b\|u\|^2)\int_{Q}\cfrac{(|u(x)|-|u(y)|)((u_k(x)+\delta)^{\beta}
			-(u_k(y)+\delta)^{\beta})}{|x-y|^{N+2s}}dxdy\\
		&\leq\frac{(\beta+1)^2}{4\beta}\int_\Omega\left(|\lambda h(x) u^{-\gamma}+\tilde{f}(x, u)|\right)\left((u_k+\delta)^{\beta}-\delta^{\beta}\right) dx\\
		&\leq\frac{(\beta+1)^2}{4\beta}\int_\Omega\left(|\lambda h(x) u^{-\gamma}|+|\tilde{f}(x, u)|\right)\left((u_k+\delta)^{\beta}-\delta^{\beta}\right) dx\\
		&=\frac{(\beta+1)^2}{4\beta}\int_{\{u\geq1\}}\left(|\lambda h(x) u^{-\gamma}|+|\tilde{f}(x, u)|\right)\left((u_k+\delta)^{\beta}-\delta^{\beta}\right) dx\\	
		&\leq\frac{(\beta+1)^2}{4\beta}\int_{\{u\geq1\}}\left(|\lambda|\|h\|_{\infty}+(|c_1|+|c_2||u|^{p})\right)\left((u_k+\delta)^{\beta}-\delta^{\beta}\right) dx\\				
		&\leq C_1\frac{(\beta+1)^2}{4\beta}\int_{\{u\geq1\}}\left(1+|u|^{p}\right)\left((u_k+\delta)^{\beta}-\delta^{\beta}\right) dx\\
		&\leq 2C_1\frac{(\beta+1)^2}{4\beta}\int_{\{u\geq1\}}|u|^{p}\left((u_k+\delta)^{\beta}-\delta^{\beta}\right) dx\\
		&\leq C\frac{(\beta+1)^2}{4\beta}|u|_{2_s^*}^{p}|(u_k+\delta)^{\beta}|_q
		\end{split}
		\end{align}
		where, $q=\frac{2_s^*}{2_s^*-p}$ and $C=\max\{1,|\lambda|\}.$ The rest of the proof is similar to the Lemma 2.7 in \cite{positivity} to obtain
		\begin{equation}\label{bound est 11}
		\|u_k\|_{\infty}\leq C\eta^{\frac{\eta}{(\eta-1)^2}}\left(|\Omega|^{1-\frac{1}{q}-\frac{2s}{N}} \right)^{\frac{\eta}{\eta-1}}\left(|(u-1)^+|_q+\delta|\Omega|^{\frac{1}{q}}\right)
		\end{equation}
		Now letting $k\rightarrow\infty$ in \eqref{bound est 11}, we have
		\begin{equation}\label{bound est 12}
		\|(u-1)^+\|_{\infty}\leq C\eta^{\frac{\eta}{(\eta-1)^2}}\left(|\Omega|^{1-\frac{1}{q}-\frac{2s}{N}} \right)^{\frac{\eta}{\eta-1}}\left(|(u-1)^+|_q+\delta|\Omega|^{\frac{1}{q}}\right)
		\end{equation}
		Hence, we conclude that $u\in L^{\infty}(\Omega).$
	\end{proof}
	
	\begin{proof}[{\bf Proof of Theorem \ref{main thm1}}]
		By using the assumption $(A1)$ and the definition of $\xi$, we get the functional $\tilde{I}$ is even and $\tilde{I}(0)=0.$ Thus, on using Theorem \ref{sym mountain}, Lemma \ref{lemma ps} and Lemma \ref{lemma genus}, we conclude that $\tilde{I}$ has sequence of critical points $\{u_n\}$ such that $\tilde{I}(u_n)<0$ and $\tilde{I}(u_n)\rightarrow0^-$.\\
		We now prove that the critical points of $\tilde{I}$ are nonnegative.\\
		{\bf Claim:} Let $u_n$ be a critical point of $\tilde{I}$, then $u_n\geq0$ a.e. in $X_0$ for every $n\in\mathbb{N}$.
		\begin{proof}
			We first divide the domain as $\Omega= \Omega^+\cup\Omega^-$, where $\Omega^+=\{x\in X_0: u_n(x)\geq0 \}$ and $\Omega^-=\{x\in X_0: u_n(x)<0 \}$. We define $u_n=u_n^+-u_n^-$, where $u_n^+(x)=\max\{u_n(x), 0\}$ and $u_n^-(x)=\max\{-u_n(x), 0\}$. We proceed through a contradiction by taking $u_n<0$ a.e. in $\Omega$. Then on choosing, $\phi=u_n^-$ as the test function in the equation \eqref{weak cutoff} in association with the inequality $(a-b)(a^--b^-)\leq-(a^--b^-)^2$, we obtain
			\begin{align*}
			&\int_{\Omega}\left(\lambda h(x) \frac{sign(u_n)u_n^-}{|u_n|^{\gamma}}+\tilde{f}(x,u_n)u_n^-\right)dx\\
			&=(a+b[u_n]^2)\int_{Q}\frac{(u_n(x)-u_n(y))(u_n^-(x)-u_n^-(y))}{|x-y|^{N+2s}}dxdy+\int_{\Omega}\phi_{u_n}u_nu_n^-dx\\
			&=-(a+b\|u_n\|^2)\|u_n^-\|^2-\int_{\Omega}\phi_{u_n}\|u_n^-\|^2dx\\
			&\Rightarrow\lambda\int_{\Omega^-}h(x)|u_n^-|^{1-\gamma}dx<0.
			\end{align*}
			Therefore, $|\Omega^-|=0$, which is a contradiction to the assumption $u_n<0$ a.e. in $\Omega$.
		\end{proof}
		\noindent We now prove $u_n\rightarrow0$ in $X_0.$ Indeed by the definition of $\tilde{I}$, we obtain
		\begin{align*}
		\frac{1}{p}\langle\tilde{I}^{'}(u_n), u_n\rangle-\tilde{I}(u_n)&=\frac{1}{p}\left[(a+b\|u_n\|^2)\|u_n\|^2+\int_{\Omega}\phi_{u_n}u_n^2-\int_{\Omega}\left(\lambda\frac{h(x)sign(u_n)u_n}{|u_n|^{\gamma}}+\tilde{f}(x,u_n)u_n\right)dx\right]\\
		&-\left[\frac{a}{2}\|u_n\|^2+\frac{b}{4}\|u_n\|^4+\frac{1}{4}\int_{\Omega}\phi_{u_n}u_n^2-\int_{\Omega}\left(\frac{\lambda h(x)}{1-\gamma}|u_n|^{1-\gamma}+\tilde{F}(x,u_n)\right)dx\right]\\
		&=a(\frac{1}{p}-\frac{1}{2})\|u_n\|^2+b(\frac{1}{p}-\frac{1}{4})\|u_n\|^4+(\frac{1}{p}-\frac{1}{4})\int_{\Omega}\phi_{u_n}u_n^2\\ &-\lambda(\frac{1}{p}-\frac{1}{1-\gamma})\int_{\Omega}h(x)|u_n|^{1-\gamma}dx+\frac{1}{p}\int_{\Omega}(p\tilde{F}(x,u_n)-\tilde{f}(x,u_n))dx\\
		&\geq a(\frac{1}{p}-\frac{1}{2})\|u_n\|^2 +b(\frac{1}{p}-\frac{1}{4})\|u_n\|^4 +\lambda(\frac{1}{1-\gamma}-\frac{1}{p})\int_{\Omega}h(x)|u_n|^{1-\gamma}dx\\
		&\geq(\frac{1}{p}-\frac{1}{2})\|u_n\|^2
		\end{align*}
		Therefore, by using the fact
		\begin{align*}
		&\frac{1}{p}\langle\tilde{I}^{'}(u_n), u_n\rangle-\tilde{I}(u_n)=o(1)\\
		&\Rightarrow(\frac{1}{p}-\frac{1}{2})\|u_n\|^2\leq o(1),
		\end{align*}
		as $n\rightarrow\infty$. Since, $1-\gamma<p<2$, we conclude that $u_n\rightarrow0$ in $X_0$. Thus from the Lemma \ref{bounded}, we can obtain $\|u_n\|_{L^{\infty}(\Omega)}\leq l$ as $n\rightarrow\infty$, thanks to the Moser iteration method. Hence, the problem \eqref{main2} has infinitely many solutions. Moreover, by using $u_n\geq0$ and $\tilde{I}(u_n)<0$, we conclude that the problem \eqref{problem main} has infinitely many weak solutions in $X_0$. Thus Theorem \ref{main thm1} is proved.
	\end{proof}

	\section{Existence of solution for $\gamma>1$.}	
		This section is fully devoted to establish the existence of a weak solution to the problem \eqref{problem main} in $X_0$. Further, we will prove that for $k\equiv0$, the problem \eqref{problem main} possesses a unique solution. Let us define the following two subsets of $X_0$ similar to the Nehari manifold.
	$$N_{1}=\{u\in X_0: (a+b[u]^{2})\|u\|^{2}+\int_{\Omega}\phi_{u}u^{2} -\lambda\int_{\Omega}h(x)|u|^{1-\gamma}-\int_{\Omega}k(x)|u|^{1+p}\geq0\},$$
	$$N_{2}=\{u\in X_0: (a+b[u]^{2})\|u\|^{2}+\int_{\Omega}\phi_{u}u^{2} -\lambda\int_{\Omega}h(x)|u|^{1-\gamma}-\int_{\Omega}k(x)|u|^{1+p}=0\}.$$
	We will show that the fractional Kirchhoff-Schr\"{o}dinger-Poisson system with a strong singularity has a weak solution in $N_2$. One can see that $N_2$ is not closed. We will prove that $N_1$ is closed in $X_0$ and the functional $I$ is coercive and bounded below on $N_1$. Further we will obtain a minimizing sequence $\{u_n\}$ of $c=\inf_{N_{1}}I$ such that $\{u_n\}$ converges to $u\in X_0$. Finally, we will show that $u\in N_2$ and hence $u$ is a weak solution to the problem \eqref{problem reduced}. We begin with the following Lemmas.
	\begin{lemma}\label{l2.2}
		Let $\int_{\Omega}\lambda h(x)|u|^{1-\gamma}<\infty$ for some $u\in X_0$. Then there exists a unique $t_0>0$ such that $t_0u\in N_{2}$ and $tu\in N_{1}$, for $t\geq t_0$, i.e. $N_{1}$, $N_{2}\neq\emptyset.$ Moreover, for $t\geq 0, \psi \in X_0$ the function  $f$ defined as $\theta(t)=t(u+t\psi)$ is continuous on $[0, \infty)$.
	\end{lemma}	
	\begin{proof}
	Let $\int_{\Omega}\lambda h(x)|u|^{1-\gamma}<\infty$ for some $u\in X_0$. Now for $t>0$, we get
	$$I(tu)= \frac{at^{2}}{2}\|u\|^{2}+\frac{bt^{4}}{4}\|u\|^{4} +\frac{t^{4}}{4}\int_{\Omega}\phi_{u}u^{2}-\frac{t^{1-\gamma}}{1-\gamma}\int_{\Omega}\lambda h(x)|u|^{1-\gamma}-\frac{t^{p+1}}{p+1}\int_{\Omega}k(x)|u|^{1+p}.$$
	It is easy to see that $tu\in N_{1} \Leftrightarrow I'(tu) \geq 0$ and $tu\in N_{2} \Leftrightarrow I'(tu) = 0$. Also, since $0<p<1<\gamma$, $I(tu)\rightarrow+\infty$, if $t\rightarrow 0^+$ as well as $t\rightarrow +\infty$ and there exists a unique $t_0>0$ such that $I'(t_0u) = 0$, $I'(tu)\geq 0$, $t\geq t_0$ and $I(t_0u) = \min_{t\geq 0}I(tu)$. Therefore, $tu\in N_{1}$, $t_0u\in N_{2}$ for $t\geq t_0$,  and $I(tu)\geq I(t_0u)$.\\
	Again, observe that for $t, \psi\geq 0$, $\int_{\Omega}\lambda h(x)|u+t\psi|^{1-\gamma}<\infty$. Now, consider a nonnegative sequence $\{t_{n}\}$ such that $t_{n}\rightarrow t$ as $n\rightarrow\infty$. On using the arguments as above, there exists $\theta(t_{n}),\theta(t)\geq 0$ such that $\theta(t_{n})(u+t_{n}\psi), \theta(t)(u+t\psi)\in N_{2}$. Thus, we get
	\begin{align}\label{2.1}
	\begin{split}
	a\theta^{1+\gamma}(t_{n})\| u+t_{n}\psi\|^{2}&+\theta^{3+\gamma}(t_{n})\left(b\| u+t_{n}\psi\|^{4}+\int_{\Omega}\phi_{u+t_{n}\psi}(u+t_{n}\psi)^{2}\right)\\&-\theta^{p+\gamma}(t_{n})\int_{\Omega}k(x)|u+t_{n}\psi|^{1+p}= \int_{\Omega}\lambda h(x)|u+t_{n}\psi|^{1-\gamma}
	\end{split}
	\end{align} and	
	\begin{align}\label{2.2}
	\begin{split}
	a\theta^{1+\gamma}(t)\|u+t\psi\|^{2}&+\theta^{3+\gamma}(t)\left(b\|u+t\psi\|^{4}+\int_{\Omega}\phi_{u+t\psi}(u+t\psi)^{2}\right)\\&-\theta^{p+\gamma}(t)\int_{\Omega}k(x)|u+t\psi|^{1+p}= \int_{\Omega}\lambda h(x)|u+t\psi|^{1-\gamma}.
	\end{split}
	\end{align}
	Now for all $n \in \mathbb{N}$, we have $\lambda h(x)|u+t_{n}\psi|^{1-\gamma}\leq \lambda h(x)|u|^{1-\gamma}$.and for each $x\in\Omega$, we have the pointwise convergence $\lambda h(x)|u+t_{n}\psi|^{1-\gamma}\rightarrow \lambda h(x)|u+t\psi|^{1-\gamma}$ as $n\rightarrow\infty$. Therefore, by using Lebesgue's dominated convergence theorem, we get $\int_{\Omega}\lambda h(x)|u+t_{n}\psi|^{1-\gamma}\rightarrow\int_{\Omega}\lambda h(x)|u+t\psi|^{1-\gamma}$ as $n\rightarrow\infty$. Further, from \eqref{2.1}, one can see that the sequence $\{\theta(t_{n})\}$ is bounded. Therefore, it has a convergent subsequence. Let $\{\theta(t_{n_k})\}$ convergence to $s$. Then, on using \eqref{2.1} and \eqref{2.2} and the above arguments, we can conclude that $s=\theta(t)$. Hence $\theta$ is continuous.
\end{proof}
	\noindent In the following Lemma, we establish that $N_1$ is closed in $X_0$ and the the functional $I$ is coercive and bounded below on $N_1$.
	\begin{lemma}\label{l2.}
		$N_1$ is closed in $X_0$ and for  all for $u\in N_{1}$, there exists $C>0$ such that $\|u\|\geq C$. Moreover, the functional $I$ is coercive and bounded below on $N_1$.
	\end{lemma}
\begin{proof}
	We first show that $N_1$ is closed. Let $\{u_{n}\}\subset N_{1}$ be such that $u_{n} \rightarrow u$ in $X_0$. Since, $\{u_{n}\}\subset N_{1}$ and $\int_{\Omega}\lambda h(x)|u_{n}|^{1-\gamma}<\infty$, then $u_{n}(x)>0$ a.e. in $\Omega$ and then up to a subsequence, $u_{n}(x)\rightarrow u(x)$ a.e. in $\Omega$. Therefore, on applying the Fatou's lemma and then using Sobolev embedding, we get
	\begin{align*}
	\int_{\Omega}\lambda h(x)|u|^{1-\gamma}&\leq\lim_{n\rightarrow\infty}\inf\int_{\Omega}\lambda h(x)|u_{n}|^{1-\gamma}\\
	&\leq\lim_{n\rightarrow\infty}\inf\left((a+b\|u_{n}\|^{2})\|u_{n}\|^{2}+\int_{\Omega}\phi_{u_{n}}u_{n}^{2}-\int_{\Omega}k(x)|u_{n}|^{1+p}\right)\\
	&\leq(a+b\| u\|^{2})\| u\|^{2}+\int_{\Omega}\phi_{u}u^{2}-\int_{\Omega}k(x)|u|^{1+p}.
	\end{align*}
	Thus we have $u\in N_{1}$ and hence $N_{1}$ is closed in $X_0$. We prove the functional $I$ is bounded below on $N_1$ by using the method of contradiction. Suppose, there exists $\{u_{n}\}\subset N_{1}$ such that $u_{n}\rightarrow 0$ in $X_0$ as $n\rightarrow\infty$. Then, on using the Reverse H\"{o}lder inequality, we get
	\begin{align*}
	(a+b\|u_{n}\|^{2})\|u_{n}\|^{2}+\int_{\Omega}\phi_{{u}_{n}}u_{n}^{2}&\geq\int_{\Omega}\lambda h(x)|u_{n}|^{1-\gamma}+\int_{\Omega}k(x)|u_{n}|^{1+p}\\
	&\geq\left(\int_{\Omega}|\lambda h(x)|^{\frac{1}{\gamma}}\right)^{\gamma}\left(\int_{\Omega}|u_{n}|\right)^{1-\gamma}\\
	&\geq c\left(\int_{\Omega}|\lambda h(x)|^{\frac{1}{\gamma}}\right)^{\gamma}\| u_{n}\|^{1-\gamma}
	\end{align*}
	where, $c$ is a positive constant. This gives a contradiction, since $\gamma>1$. Therefore, there exists $C>0$ such that $\|u\|\geq C$ for all $u\in N_{1}$.\\
	Since, $u\in N_{1}$ implies that $\int_{\Omega}\lambda h(x)|u|^{1-\gamma}\leq(a+b\|u\|^{2})\|u\|^{2} +\int_{\Omega}\phi_{u}u^{2}-\int_{\Omega}k(x)u^{1+p}<\infty$. Therefore, from the definition \eqref{functional} of $I$, we have \begin{align}\label{2.3}
	I(u)&=\frac{a}{2}\|u\|^{2}+\frac{b}{4}\|u\|^{4}+\frac{1}{4}\int_{\Omega}\phi_{u}u^{2}-\frac{1}{1-\gamma}\int_{\Omega}\lambda h(x)|u|^{1-\gamma}-\frac{1}{p+1}\int_{\Omega}k(x)|u|^{1+p}\nonumber\\
	&\geq\frac{a}{2}\| u\|^{2}+\frac{b}{4}\| u\|^{4}-c|k|_{\infty}\| u\|^{1+p}.
	\end{align}
	Now, since $0<p<1$ and $a+b\geq 0$, therefore by using the Lemma \ref{l2.2}, we get that $I$ is coercive and bounded below on $N_{1}.$
	\end{proof}
	\begin{lemma}\label{lemma min-seq}
		Assume the compatibility condition \eqref{compatibility} holds. Then there exists a minimizing sequence $\{u_{n}\}\subset N_{1}$ of $c=\inf\limits_{N_{1}}I$, i.e. there exists $u\in N_{1}$ such that $I(u)=c$. Moreover,  $u\in N_{2}$.
	\end{lemma}
\begin{proof}
	It is easy to see that the functional $I$, defined as in \eqref{functional} is lower semicontinuous. Since, $N_1 (\neq\emptyset)$ is closed, then by using Ekeland's variational principle for $c=\inf_{N_{1}}I$, we can extract a minimizing sequence $\{u_{n}\}\subset N_{1}$ such that
\begin{enumerate}[label=(\roman*)]
	\item $I(u_{n}) \leq\ c+\frac{1}{n^{2}}$;
	\item $I(u_{n}) \leq I(v)+\frac{1}{n}\| u_{n}-v\|$, $\forall\, v\in N_{1}.$
\end{enumerate}
Now, from the fact $I(|u|) =I(u)$, one can assume that $u_{n} >0$ a.e. in $\Omega$. By Lemma \ref{l2.}, on using the coerciveness of $I$, we have $\{u_{n}\}$ is bounded. Therefore, upto to a subsequence, we have
\begin{enumerate}[label=(\roman*)]
	\item $u_{n}\rightharpoonup u$ weakly in $X_0$
	\item $u_{n}\rightarrow u$ strongly in $L^{q}(\Omega)$ for $q\in [1, 2_s^*)$, $2_s^*=\frac{2N}{N-2s}$, and
	\item $u_{n}(x) \rightarrow u(x)$ pointwise a. e. in $\Omega$.
\end{enumerate}
Since, $N_1$ is closed and $\gamma>1$, then from Fatou's lemma and $\{u_{n}\}\subset N_{1}$, we get $u>0$ a. e. in $\Omega$, $\int_{\Omega}\lambda h(x)|u|^{1-\gamma}<\infty$ and $u\in N_1$. On using Fatou's lemma and Lemma \ref{l2.2}, we have
\begin{align}\label{strong sub}
	\inf_{N_{1}}I&=\liminf_{n\rightarrow\infty} I(u_{n})\nonumber\\
	&=\liminf_{n\rightarrow\infty}\left(\frac{a}{2}\|u_{n}\|^{2}+\frac{b}{4}\|u_{n}\|^{4}+\frac{1}{4}\int_{\Omega}\phi_{u_{n}}u_{n}^{2}-\frac{1}{1-\gamma}\int_{\Omega}\lambda h(x)u_{n}^{1-\gamma}-\frac{1}{1+p}\int_{\Omega}k(x)u_{n}^{1+p}\right)\nonumber\\
	&\geq\liminf_{n\rightarrow\infty}\frac{a}{2}\|u_{n}\|^{2}+\frac{b}{4}\liminf_{n\rightarrow\infty}\|u_{n}\|^{4}+\frac{1}{4}\phi_{u}u^{2}-\frac{1}{1-\gamma}\liminf_{n\rightarrow\infty}\int_{\Omega}\lambda h(x)u_{n}^{1-\gamma}-\frac{1}{1+p}\int_{\Omega}k(x){u}^{1+p}\nonumber\\
	&\geq\frac{a}{2}\|u\|^{2}+\frac{b}{4}\|u\|^{4}+\frac{1}{4}\int_{\Omega}\phi_{u}u^{2}-\frac{1}{1-\gamma}\int_{\Omega}\lambda h(x)u^{1-\gamma}-\frac{1}{p+1}\int_{\Omega}k(x)u^{1+p}\nonumber\\
	&=I(u)\geq I(t_0u)\nonumber\\
	&\geq\inf_{N_{2}}I\geq\inf_{N_{1}}I.
\end{align}
Hence, $t_0=1$, i.e. $u\in N_2$ and hence $I(u)=c$. This completes the proof.
\end{proof}
\noindent{\bf Proof of Theorem \ref{main theorem}.}\\
	Suppose $u$ is a solution to the problem \eqref{problem main}, then the compatibility condition \eqref{compatibility} must be true. We will prove the other part. Let \eqref{compatibility} be true. We first prove the following inequality which is essential to guarantee the existence of solution.
		\begin{align}\label{1.7}
		(a+b\|u\|^{2})\int_{\Omega}(-\Delta)^{s/2} u\cdot(-\Delta)^{s/2}\psi+\int_{\Omega}\phi_{u}u\psi -\int_{\Omega}k(x)u^{p}\psi\geq\int_{\Omega}\lambda h(x)u^{-\gamma}\psi.
		\end{align}
for every nonnegative $\psi\in X_0$. We will divide the proof of \eqref{1.7} in two cases, i.e. either $\{u_{n}\}\subset N_{1}\setminus N_{2}$ or $\{u_{n}\}\subset N_{2}$.\\	
		{\bf Case 1.} {\it $\{u_{n}\}\subset N_{1}\setminus N_{2}$ for $n$ large enough.}\\
		For a given nonnegative function $\psi\in X_0$, by $\{u_{n}\}\subset N_{1}\setminus N_{2}$, we derive that
		\begin{equation*}
		(a+b\| u_{n}\|^{2})\| u_{n}\|^{2}+\int_{\Omega}\phi_{u_{n}}u_{n}^{2}-\int_{\Omega}k(x)|u_{n}|^{1+p}>\int_{\Omega}\lambda h(x)u_{n}^{1-\gamma}\geq\int_{\Omega}\lambda h(x)(u_{n}+t\psi)^{1-\gamma},\ t\geq 0,
		\end{equation*}
		then by the continuity, there exists $t>0$ small enough such that
		\begin{equation*}
		(a+b\| u_{n}+t\psi\|^{2})\| u_{n}+t\psi\|^{2}+\int_{\Omega}\phi_{u_{n}+t\psi}(u_{n}+t\psi)^{2}-\int_{\Omega}k(x)(u_{n}+t\psi)^{1+p}\geq\int_{\Omega}\lambda h(x)(u_{n}+t\psi)^{1-\gamma},	
		\end{equation*}
		that is $(u_{n}+t\psi)\in N_{1}$. Then, by (ii) of Ekeland's variational principle, we have
		\begin{equation*}
		\frac{1}{n}\| t\psi\|+I(u_{n}+t\psi)-I(u_{n})\geq 0.
		\end{equation*}
		That is,
		\begin{align*}
		\frac{\|t\psi\|}{n}&+\frac{a}{2}\left(\|u_{n}+t\psi\|^{2}-\|u_{n}\|^{2}\right)+\frac{b}{4}\left(\|u_{n}+t\psi\|^{4}-\|u_{n}\|^{4}\right)\\
		&+\frac{1}{4}\int_{\Omega}\left(\phi_{u_{n}+t\psi}(u_{n}+t\psi)^{2}-\phi_{u_{n}}u_{n}^{2}\right)-\frac{1}{1+p}\int_{\Omega}k(x)\left((u_{n}+t\psi)^{1+p}-u_{n}^{1+p}\right)\\
		&\geq\frac{1}{1-\gamma}\int_{\Omega}\lambda h(x)\left((u_{n}+t\psi)^{1-\gamma}-u_{n}^{1-\gamma}\right).
		\end{align*}
		Dividing by $t>0$ and by Fatou's lemma, we conclude that
		\begin{align*}
		\frac{1}{n}\|\psi\|+(a+b\|u_{n}\|^{2})\int_{\Omega}(-\Delta)^{s/2} u_{n}\cdot(-\Delta)^{s/2}\psi &+\int_{\Omega}\phi_{u_{n}}u_{n}\psi-\int_{\Omega}k(x)u_{n}^{p}\psi\\
		&\geq\lim_{t\rightarrow 0} \inf\int_{\Omega}\frac{\lambda h(x)\left((u_{n}+t\psi)^{1-\gamma}-u_{n}^{1-\gamma}\right)}{(1-\gamma)t}\\
		&\geq\int_{\Omega}\lambda h(x)u_{n}^{-\gamma}\psi,
		\end{align*}
		Now by Lemma \ref{lemma min-seq}, we have $I(u)=c$ for some $u\in N_2$. Therefore, by using \eqref{strong sub}, we obtain $\|u_n\|^2\rightarrow\|u\|^2$ for every $a>0, b\geq0$ and similarly, $\|u_n\|^4\rightarrow\|u\|^4$ for every $b>0$ with $a=0$. In both of the cases, $\|u_n\|\rightarrow\|u\|$ as $n\rightarrow\infty$. Thus, by applying Fatou's lemma once again, we get
		\begin{equation*}
		\left(a+b\| u\|^{2}\right)\int_{\Omega}(-\Delta)^{s/2} u\cdot(-\Delta)^{s/2}\psi+\int_{\Omega}\phi_{u}u\psi -\int_{\Omega}k(x)u^{p}\psi\geq\int_{\Omega}\lambda h(x)u^{-\gamma}\psi.
		\end{equation*}
		{\bf Case 2.} {\it There exists a subsequence of $\{u_{n}\}$ (still denoted by $\{u_{n}\}$) belonging to $N_{2}$.}\\
		In this case, we can also show that \eqref{1.7} holds. For given nonnegative $\psi\in X_0$, for each $u_{n}\in N_{2}$ and $t\geq 0$,
		\begin{equation*}
		\int_{\Omega}\lambda h(x)(u_{n}+t\psi)^{1-\gamma}\leq\int_{\Omega}\lambda h(x)u_{n}^{1-\gamma}<\infty.
		\end{equation*}
		By Lemma \ref{l2.2}, there exists $t(u_{n}+t\psi)>0$ satisfying $t(u_{n}+t\psi)(u_{n}+t\psi)\in N_{2}$. For clarity, we denote $\theta_{n}(t)=t(u_{n}+t\psi)$, it is obvious that $\theta_{n}(0)=1$. By $\theta_{n}(t)(u_{n}+t\psi)\in N_{2}$, we have
		\begin{align}\label{2.4}
		a\theta_{n}^{2}(t)\|u_{n}+t\psi\|^{2}&+b\theta_{n}^{4}(t)\|u_{n}+t\psi\|^{4}+\theta_{n}^{4}(t)\int_{\Omega}(\phi_{u_{n}+t\psi}(u_{n}+t(\psi)^{2}\nonumber\\
		&-\theta_{n}^{1-\gamma}(t)\int_{\Omega}\lambda h(x)(u_{n}+t\psi)^{1-\gamma}-\theta_{n}^{1+p}(t)\int_{\Omega}k(x)(u_{n}+t\psi)^{1+p}=0. 
		\end{align}
		By Lemma \ref{l2.2}, for given $n$, $\theta_{n}$ is continuous on $[0, \infty)$. We denote $D_{+}\theta_{n}(0)$ the right lower Dini derivative of $\theta_{n}$ at zero. Next, we shall show that $\theta_{n}$ has uniform behavior at zero with respect to $n$, i.e., $|D_{+}\theta_{n}(0)|\leq{C}$ for suitable $C>0$ independent of $n$. By the definition of $D_{+}\theta_{n}(0)= \lim_{t\rightarrow 0^+}\inf\frac{\theta_{n}(t)-\theta_{n}(0)}{t}$, there exists a sequence $\{t_{k}\}$ with $t_{k}>0$ and $t_{k}\rightarrow 0$ as $k\rightarrow\infty$ such that
		\begin{equation*}
		D_{+}\theta_{n}(0)=\lim_{k\rightarrow\infty}\frac{\theta_{n}(t_{k})-\theta_{n}(0)}{t_{k}}.
		\end{equation*}
		By $u_{n}\in N_{2}$ and \eqref{2.4}, for $t>0$, we get that
		\begin{align}\label{2.5}
		0=\frac{1}{t}&\left[a\left(\theta_{n}^{2}(t)-1\right)\|u_{n}+t\psi\|^{2}+\left(\theta_{n}^{4}(t)-1\right)\left(b\|u_{n}+t\psi\|^{4}+\int_{\Omega}\phi_{u_{n}+t\psi}(u_{n}+t\psi)^{2}\right)\right.\nonumber\\
		&-\left(\theta_{n}^{1-\gamma}(t)-1\right)\int_{\Omega}\lambda h(x)(u_{n}+t\psi)^{1-\gamma}-\left(\theta_{n}^{1+p}(t)-1\right)\int_{\Omega}k(x)(u_{n}+t\psi)^{1+p}\nonumber\\
		&+a\left(\|u_{n}+t\psi\|^{2}-\|u_{n}\|^{2}\right)+b\left(\|u_{n}+t\psi\|^{4}-\|u_{n}\|^{4}\right)+\int_{\Omega}(\phi_{u_{n}+t\psi}(u_{n}+t\psi)^{2}-\phi_{u_{n}}u_{n}^{2})\nonumber\\
		&\left.-\int_{\Omega}\lambda h(x)\left((u_{n}+t\psi)^{1-\gamma}-u_{n}^{1-\gamma}\right)-\int_{\Omega}k(x)\left((u_{n}+t\psi)^{1+p}-u_{n}^{1+p}\right)\right]\nonumber\\
		&\geq\frac{\theta_{n}(t)-1}{t}\left[a(\theta_{n}(t)+1)\|u_{n}+t\psi\|^{2}+\frac{\theta_{n}^{4}(t)-1}{\theta_{n}(t)-1}\left(b\|u_{n}+t\psi\|^{4}+\int_{\Omega}\phi_{u_{n}+t\psi}(u_{n}+t\psi)^{2}\right)\right.\nonumber\\
		&\left.-\frac{\theta_{n}^{1-\gamma}(t)-1}{\theta_{n}(t)-1}\int_{\Omega}\lambda h(x)(u_{n}+t\psi)^{1-\gamma}-\frac{\theta_{n}^{1+p}(t)-1}{\theta_{n}(t)-1}\int_{\Omega}k(x)(u_{n}+t\psi)^{1+p}\right]\nonumber\\
		&+\frac{1}{t}\left[a(\|u_{n}+t\psi\|^{2}-\|u_{n}\|^{2})+b(\|u_{n}+t\psi\|^{4}-\|u_{n}\|^{4})+\int_{\Omega}(\phi_{u_{n}+t\psi}(u_{n}+t\psi)^{2}-\phi_{u_{n}}u_{n}^{2})\right.\nonumber\\
		&\left.-\int_{\Omega}k(x)\left((u_{n}+t\psi)^{1+p}-u_{n}^{1+p}\right)\right],
		\end{align}
		replacing $t$ in \eqref{2.5} with $t_{k}$ and letting $k\rightarrow\infty$, we deduce that
		\begin{align*}
		0&\geq{D}_{+}\theta_{n}(0)\left[2a\|u_{n}\|^{2}+4\left(b\|u_{n}\|^{4}+\int_{\Omega}\phi_{u_{n}}u_{n}^{2}\right)-(1-\gamma)\int_{\Omega}\lambda h(x)u_{n}^{1-\gamma}-(1+p)\int_{\Omega}k(x)u_{n}^{1+p}\right]\\
		&+(2a+4b\|u_{n}\|^{2})\int_{\Omega}(-\Delta)^{s/2}{u}_{n}\cdot(-\Delta)^{s/2}\psi+4\int_{\Omega}\phi_{u_{n}}u_{n}\psi-(p+1)\int_{\Omega}k(x)u_{n}^{p}\\
		&=D_{+}\theta_{n}(0)\left[a(1-p)\|u_{n}\|^{2}+(3-p)\left(b\|u_{n}\|^{4}+\int_{\Omega}\phi_{u_{n}}u_{n}^{2}\right)+(p+\gamma)\int_{\Omega}\lambda h(x)u_{n}^{1+\gamma}\right]\\
		&+(2a+4b\|u_{n}\|^{2})\int_{\Omega}(-\Delta)^{s/2}{u}_{n}\cdot(-\Delta)^{s/2}\psi+4\int_{\Omega}\phi_{u_{n}}u_{n}\psi-(p+1)\int_{\Omega}k(x)u_{n}^{p}\psi.
		\end{align*}
		By Lemma \ref{l2.2}, we get that
		\begin{align*}
		D_{+}\theta_{n}(0)(a(1-p)\alpha^{2}+(3-p)b\alpha^{4})&+(2a+4b\|u_{n}\|^{2})\int_{\Omega}(-\Delta)^{s/2} u_{n}\cdot(-\Delta)^{s/2}\psi\\
		&+4\int_{\Omega}\phi_{u_{n}}u_{n}\psi-(p+1)\int_{\Omega}k(x)u_{n}^{p}\psi\leq 0.
		\end{align*}
		Since, $p\in(0,1)$, this implies that $D_{+}\theta_{n}(0)\neq+\infty$ and $D_{+}\theta_{n}(0)$ is bounded from above uniformly in $n$. That is $D_{+}f_{n}(0)\in[-\infty, +\infty$) and
		\begin{equation}\label{2.6}
		D_{+}\theta_{n}(0)\leq C_{1} ~\text{uniformly for}~ n
		\end{equation}	
		for some $C_{1}>0.$\\
		On the other hand, we can obtain the lower bound for $D_{+}\theta_{n}(0)$. If $D_{+}\theta_{n}(0)\geq 0$ for $n$ large, this gives the results. Otherwise, up to a subsequence, still denoted by $D_{+}\theta_{n}(0)$ such that $D_{+}\theta_{n}(0)$ are negative (possibly $-\infty$). Then by (ii) of Ekeland's variational principle, for $t>0$,	we have
		\begin{align*}
		&\frac{(1-\theta_{n}(t))\|u_{n}\|+t\theta_{n}(t)\|\psi\|}{n}\\
		&\geq\frac{\|u_{n}-\theta_{n}(t)(u_{n}+t\psi)\|}{n}\\
		&\geq I(u_{n})-I(f_{n}(t)(u_{n}+t\psi))\\
		&=\frac{a(\gamma+1)}{2(\gamma-1)}(\|u_{n}\|^{2}-\|u_{n}+t\psi\|^{2})+\frac{\gamma+3}{4(\gamma-1)}\left[b(\|u_{n}\|^{4}-\|u_{n}+t\psi\|^{4})+\int_{\Omega}(\phi_{u_{n}}u_{n}^{2}-\phi_{u_{n}+t\psi}(u_{n}+t\psi)^{2})\right]\\
		&-\frac{p+\gamma}{(\gamma-1)(p+1)}\int_{\Omega}k(x)\left(u_{n}^{p+1}-(u_{n}+t\psi)^{p+1}\right)-\frac{a(\gamma+1)}{2(\gamma-1)}(f_{n}^{2}(t)-1)\|u_{n}+t\psi\|^{2}\\
		&-\frac{\gamma+3}{4(\gamma-1)}(f_{n}^4(t)-1)\left(b\|u_{n}+t\psi\|^{4}+\int_{\Omega}(\phi_{u_{n}+t\psi}(u_{n}+t\psi)^{2})\right)\\
		&+\frac{p+\gamma}{(\gamma-1)(p+1)}(f_{n}^{p+1}(t)-1)\int_{\Omega}k(x)(u_{n}+t\psi)^{p+1}.
		\end{align*}
		Then,
		\begin{align}\label{2.7}
		\frac{t\theta_{n}(t)}{n}\|\psi\|&\geq{I}(u_{n})-I(\theta_{n}(t)(u_{n}+t\psi))+\frac{\theta_{n}(t)-1}{n}\|u_{n}\|\nonumber\\
		&=(\theta_{n}(t)-1)\left[\frac{\|u_{n}\|}{n}-\frac{a(\gamma+1)}{2(\gamma-1)}(\theta_{n}(t)+1)\|u_{n}+t\psi\|^{2}\right.\nonumber\\
		&-\frac{\gamma+3}{4(\gamma-1)}\frac{\theta_{n}^{4}(t)-1}{\theta_{n}(t)-1}\left(b\|u_{n}+t\psi\|^{4}+\int_{\Omega}\phi_{u_{n}+t\psi}(u_{n}+t\psi)^{2}\right)\nonumber\\
		&\left.+\frac{p+\gamma}{(\gamma-1)(p+1)}\frac{\theta_{n}^{p+1}(t)-1}{\theta_{n}(t)-1}\int_{\Omega}k(x)(u_{n}+t\psi)^{p+1}\right]+\frac{a(\gamma+1)}{2(\gamma-1)}(\|u_{n}\|^{2}-\|u_{n}+t\psi\|^{2})\nonumber\\
		&+\frac{\gamma+3}{4(\gamma-1)}\left[b(\|u_{n}\|^{4}-\|u_{n}+t\psi\|^{4})+\int_{\Omega}(\phi_{u_{n}}u_{n}^{2}-\phi_{u_{n}+t\psi}(u_{n}+t\psi)^{2})\right]\nonumber\\
		&-\frac{p+\gamma}{(\gamma-1)(p+1)}\int_{\Omega}k(x)\left(u_{n}^{p+1}-(u_{n}+t\psi)^{p+1}\right)
		\end{align}
		Then, replacing $t$ in \eqref{2.7}, dividing $t_{k}$ and letting $k\rightarrow\infty$, we deduce that
		\begin{align}\label{2.8}
		\frac{\|\psi\|}{n}&\geq{D}_{+}\theta_{n}(0)\left[\frac{\|u_{n}\|^{2}}{n}-\frac{a(\gamma+1)}{\gamma-1}\|u_{n}\|^{2}-\frac{\gamma+3}{\gamma-1}\left(b\|u_{n}\|^{4}+\int_{\Omega}\phi_{u_{n}}u_{n}^{2}\right)+\frac{p+\gamma}{\gamma-1}\int_{\Omega}k(x)u_{n}^{p+1}\right]\nonumber\\
		&-\frac{a(\gamma+1)}{\gamma-1}\int_{\Omega}(-\Delta)^{s/2}{u}_{n}\cdot(-\Delta)^{s/2}\psi-\frac{\gamma+3}{\gamma-1}\left(b\|u_{n}\|^{2}\int_{\Omega}(-\Delta)^{s/2}{u}_{n}\cdot(-\Delta)^{s/2}\psi+\int_{\Omega}\phi_{u_{n}}u_{n}\psi\right)\nonumber\\
		&+\frac{p+\gamma}{\gamma-1}\int_{\Omega}k(x)u_{n}^{p}\psi.
		\end{align}
		Since
		\begin{align*}
		&-\frac{a(\gamma+1)}{\gamma-1}\|u_{n}\|^{2}-\frac{\gamma+3}{\gamma-1}\left(b\|u_{n}\|^{4}+\int_{\Omega}\phi_{u_{n}}u_{n}^{2}\right)+\frac{p+\gamma}{\gamma-1}\int_{\Omega}k(x)u_{n}^{p+1}\\
		&=-\frac{1}{\gamma-1}\left[a(1-p)\|u_{n}\|^{2}+(3-p)\left(b\|u_{n}\|^{2}+\int_{\Omega}\phi_{u_{n}}u_{n}^{2}\right)+(\gamma+p)\int_{\Omega}\lambda h(x)u_{n}^{1-\gamma}\right]\\
		&\leq-\frac{a(1-p)}{\gamma-1}\|u_{n}\|^{2}\\
		&\leq-\frac{a(1-p)}{\gamma-1}C^{2}.
		\end{align*}
		So, from \eqref{2.8}, we have
		\begin{align}\label{2.9}
		\frac{\|\psi\|}{n}&\geq{D}_{+}\theta_{n}(0)\left(\frac{\|u_{n}\|^{2}}{n}-\frac{(1-p)a\alpha^{2}}{\gamma-1}\right)-\frac{a(\gamma+1)}{\gamma-1}\int_{\Omega}(-\Delta)^{s/2}{u}_{n}\cdot(-\Delta)^{s/2}\psi\nonumber\\
		&-\frac{\gamma+3}{\gamma-1}\left(b\|u_{n}\|^{2}\int_{\Omega}(-\Delta)^{s/2}{u}_{n}\cdot(-\Delta)^{s/2}\psi+\int_{\Omega}\phi_{u_{n}}u_{n}\psi\right)+\frac{p+\gamma}{\gamma-1}\int_{\Omega}k(x)u_{n}^{p}\psi.
		\end{align}
		We choose $n$ large enough such that $\frac{\|u_{n}\|^{2}}{n} -\frac{(1-p)aC^{2}}{\gamma-1} <0$, we know from \eqref{2.9} that $D_{+}\theta_{n}(0)\neq-\infty$ as $n\rightarrow\infty$. That is $D_{+}\theta_{n}(0)$ is bounded from below uniformly for $n$ large enough. Hence from \eqref{2.6}, we have
		\begin{equation*}
		|D_{+}\theta_{n}(0)|\leq{C} ~\text{for $n$ large enough},
		\end{equation*}	
		for some $C>0$. Again, by (ii) of Ekeland's variation principle, we also have
		\begin{align*}
		&\frac{|\theta_{n}(t)-1|\|u_{n}\|+|t\theta_{n}(t)|\|\psi\|}{n}\\
		&\geq\frac{\|u_{n}-\theta_{n}(t)(u_{n}+t\psi)\|}{n}\\
		&\geq I(u_{n})-I(\theta_{n}(t)(u_{n}+t\psi))\\
		&=\frac{a}{2}(\|u_{n}\|^{2}-\|u_{n}+t\psi\|^{2})+\frac{1}{4}\left[b(\|u_{n}\|^{4}-\|u_{n}+t\psi\|^{4})+\int_{\Omega}(\phi_{u_{n}}u_{n}^{2}-\phi_{u_{n}+t\psi}(u_{n}+t\psi)^{2})\right]\\
		&+\frac{1}{\gamma-1}\int_{\Omega}\lambda h(x)(u_{n}^{1-\gamma}-(u_{n}+t\psi)^{1-\gamma})-\frac{1}{p+1}\int_{\Omega}k(x)(u_{n}^{p+1}-(u_{n}+t\psi)^{p+1})\\
		&+\frac{a}{2}(1-\theta_{n}^{2}(t))\|u_{n}+t\psi\|^{2}+\frac{1}{4}(1-\theta_{n}^{4}(t))\left(b\|u_{n}+t\psi\|^{4}+\int_{\Omega}\phi_{u_{n}+t\psi}(u_{n}+t\psi)^{2}\right)\\
		&+\frac{1}{\gamma-1}(1-\theta_{n}^{1-\gamma}(t))\int_{\Omega}\lambda h(x)(u_{n}+t\psi)^{1-\gamma}-\frac{1}{p+1}(1-\theta_{n}^{p+1}(t))\int_{\Omega}k(x)(u_{n}+t\psi)^{p+1}.
		\end{align*}
		The above inequality also holds for $t=t_{k}$, then dividing by $t_{k}>0$ and passing to the limit as $k\rightarrow\infty$, then by $u_{n}\in N_{2}$, we obtain that
		\begin{align*}
		\frac{|D_{+}\theta_{n}(0)|\|u_{n}\|}{n}+\frac{\|\psi\|}{n}&\geq-(a+b\|u_{n}\|^{2})\int_{\Omega}(-\Delta)^{s/2}{u}_{n}\cdot(-\Delta)^{s/2}\psi-\int_{\Omega}\phi_{u_{n}}u_{n}\psi+\int_{\Omega}k(x)u_{n}^{p}\psi\\
		&+\lim_{k\rightarrow\infty}\inf\frac{1}{\gamma-1}\int_{\Omega}\frac{\lambda h(x)(u_{n}^{1-\gamma}-(u_{n}+t_{k}\psi)^{1-\gamma})}{t_{k}}.
		\end{align*}
		Again, proceeding to the similar arguments as in {\bf Case 1}, we can obtain $\|u_n\|\rightarrow\|u\|$ as $n\rightarrow\infty$. Hence, by using $|D_{+}f_{n}(0)|\leq C$, Fatou's lemma and the strong convergence, we obtain $\int_{\Omega}\lambda h(x)u^{-\gamma}\psi<\infty$ and
		\begin{equation*}
		(a+b\|u\|^{2})\int_{\Omega}(-\Delta)^{s/2}{u}\cdot(-\Delta)^{s/2}\psi+\int_{\Omega}\phi_{u}u\psi-\int_{\Omega}\lambda h(x)u^{-\gamma}\psi-\int_{\Omega}k(x)u^{p}\psi\geq 0.
		\end{equation*}
		Thus from {\bf Case 1} and {\bf Case 2}, we obtain that the above inequality holds for $\psi\in X_0$ with $\psi\geq 0$, that is the inequality \eqref{1.7} holds.\\
		We now prove that $u$ is a weak solution to the system \eqref{problem reduced} by using$u\in N_2$ and the inequality \eqref{1.7}. For $t>0$ and $\psi\in X_0$, we define
		$$\Omega_1=\{x\in\Omega:u(x)+t\psi(x)\geq0\}~\text{and}~\Omega_2=\{x\in\Omega:u(x)+t\psi(x)<0\}.$$
		Now on using $\psi_t=(u+t\psi)^+$ as the test function in \eqref{1.7}, we get
	{\small \begin{align*}
	0&\leq(a+b\|u\|^{2})\int_{\Omega}(-\Delta)^{s/2}{u}\cdot(-\Delta)^{s/2}\psi_t+\int_{\Omega}\phi_{u}u\psi_t-\int_{\Omega}\lambda h(x)u^{-\gamma}\psi_t-\int_{\Omega}k(x)u^{p}\psi_t\\
	&=(a+b\|u\|^{2})\int_{\Omega_{1}}(-\Delta)^{s/2}{u}\cdot(-\Delta)^{s/2}(u+t\psi)+\int_{\Omega_{1}}\phi_{u}u(u+t\psi)\\&\hspace{3cm}-\int_{\Omega_{1}}\lambda h(x)u^{-\gamma}(u+t\psi)-\int_{\Omega_{1}}k(x)u^{p}(u+t\psi)\\
	&=(a+b\|u\|^{2})\int_{\Omega}(-\Delta)^{s/2}{u}\cdot(-\Delta)^{s/2}(u+t\psi)+\int_{\Omega}\phi_{u}u(u+t\psi)\\
	&\hspace{3cm}-\int_{\Omega}\lambda h(x)u^{-\gamma}(u+t\psi) -\int_{\Omega}k(x)u^{p}(u+t\psi)\\
	&-\left[(a+b\|u\|^{2})\int_{\Omega_{2}}(-\Delta)^{s/2}{u}\cdot(-\Delta)^{s/2}(u+t\psi)+\int_{\Omega_{2}}\phi_{u}u(u+t\psi)\right.\\
	&\hspace{3cm}\left.-\int_{\Omega_{2}}\lambda h(x)u^{-\gamma}(u+t\psi) -\int_{\Omega_{2}}k(x)u^{p}(u+t\psi)\right]\\
	&\leq t\left[(a+b\|u\|^{2})\int_{\Omega}(-\Delta)^{s/2}{u}\cdot(-\Delta)^{s/2}\psi+\int_{\Omega}\phi_{u}u\psi-\int_{\Omega}\lambda h(x)u^{-\gamma}\psi-\int_{\Omega}k(x)u^{p}\psi\right.\\
	&\hspace{3cm}\left.-(a+b\|u\|^{2})\int_{\Omega_{2}}(-\Delta)^{s/2}{u}\cdot(-\Delta)^{s/2}\psi-\int_{\Omega_{2}}\phi_{u}u\psi\right]	
	\end{align*}}
	Since, $u>0$ almost everywhere in $\Omega$ and the measure of $\Omega_{2}$ tends to zero as $t\rightarrow 0$, then dividing by $t>0$ and letting $t\rightarrow 0$, we obtain that
	\begin{equation*}
	(a+b\|u\|^{2})\int_{\Omega}(-\Delta)^{s/2}{u}\cdot(-\Delta)^{s/2}\psi+\int_{\Omega}\phi_{u}u\psi-\int_{\Omega}\lambda h(x)u^{-\gamma}\psi-\int_{\Omega}k(x)u^{p}\psi\geq 0, ~\psi\in X_0.
	\end{equation*}
	This inequality also holds for $-\psi$, so we have
	\begin{equation*}
	(a+b\|u\|^{2})\int_{\Omega}(-\Delta)^{s/2}{u}\cdot(-\Delta)^{s/2}\psi+\int_{\Omega}\phi_{u}u\psi-\int_{\Omega}\lambda h(x)u^{-\gamma}\psi-\int_{\Omega}k(x)u^{p}\psi=0, ~\psi\in X_0.
	\end{equation*}
	Thus, $u\in X_0$ is a solution of system \eqref{problem reduced}.\\
	\noindent {\it\bf Uniqueness of solution.} Assume the compatibility condition \eqref{compatibility} holds. Let $u, v\in X_0$ be two weak solutions to system \eqref{problem reduced}. Then from Definition \ref{weak solution defn}, we have
	\begin{equation}\label{2.10}
		(a+b\|u\|^{2})\int_{\Omega}(-\Delta)^{s/2}{u}\cdot(-\Delta)^{s/2}(u-v)+\int_{\Omega}\phi_{u}u(u-v)=\int_{\Omega}\lambda h(x)u^{-\gamma}(u-v)+\int_{\Omega}f(x,u)(u-v).
	\end{equation}
	and
	\begin{equation}\label{2.11}
		(a+b\|v\|^{2})\int_{\Omega}(-\Delta)^{s/2}{v}\cdot(-\Delta)^{s/2}(u-v)+\int_{\Omega}\phi_{v}v(u-v)=\int_{\Omega}\lambda h(x)v^{-\gamma}(u-v)+\int_{\Omega}f(x,u)(u-v).
	\end{equation}
	On subtracting \eqref{2.11} from \eqref{2.10}, we get
	\begin{align}\label{unique}
	\begin{split}
		&a\|u-v\|^{2}+b(\|u\|^{4}+\|v\|^{4}-\|u\|^{2}(u,v)-\|v\|^{2}(u,v))+\int_{\Omega}(\phi_{u}u-\phi_{v}v)(u-v)\\
		&=\int_{\Omega}\lambda h(x)(u^{-\gamma}-v^{-\gamma})(u-v) +\int_{\Omega}(f(x,u)-f(x,v))(u-v)\\
		&=\int_{\Omega}\lambda h(x)(u^{-\gamma}-v^{-\gamma})(u-v),~(\text{by using}~f(x;\cdot)\equiv0).
		\end{split}
	\end{align}
	Now, on applying H\"{o}lder's inequality, we have
	\begin{equation*}
	\|u\|^{4}+\|v\|^{4}-\|u\|^{2}(u,v)-\|v\|^{2}(u,v)\geq(\|u\|-\|v\|)^2(\|u\|^2+\|u\|\|v\|+\|v\|^2)\geq 0.
	\end{equation*}
	and $\int_{\Omega}\lambda h(x)(u^{-\gamma}-v^{-\gamma})(u-v)\leq 0$, since $\gamma>0$. Therefore, on using $a, b\geq0$ with $a+b>0$ and from Lemma \ref{phi prop}, we deduce that $\|u-v\|^{2}\leq0$.
	Hence, the solution to the system \eqref{problem reduced} is unique.
	\begin{remark}
		Observe that, if we replace $f$ by $-f$ in the problem \eqref{problem reduced}, then from \eqref{unique} one can conclude that the problem \eqref{problem reduced} possesses a unique solution.
	\end{remark}
	\begin{remark}
		The existence of solution to the problem \eqref{problem reduced}, for $\gamma=1$ can be obtained as a limit of the following sequence of problems
		\begin{align*}
		\begin{split}
		\left(a+b[u]^2\right)(-\Delta)^{s} u+\phi_u u&=\lambda\frac{h(x)}{(u+\frac{1}{n})}+f(x,u)~\text{in}~\Omega,\\
		u&>0~\text{in}~\Omega,\\
		u&=0~\text{in}~\mathbb{R}^N\setminus\Omega,
		\end{split}
		\end{align*}
	\end{remark}

	\section*{Acknowledgement}
	The author thanks the Council of Scientific and Industrial Research (CSIR), India, for the financial assistantship received to carry out this research work (CSIR no. 09/983(0013)/2017-EMR-I). The author also thanks Prof. D. Choudhuri for the suggetions and discussions.
	
	\bibliographystyle{plain}

\end{document}